\documentclass[a4paper]{article}

\usepackage[utf8]{inputenc}
\usepackage{lmodern}
\usepackage{amssymb,amsfonts,amsthm,amsmath,bbm,mathrsfs,aliascnt,mathtools}
\usepackage[english]{babel}
\usepackage[colorlinks]{hyperref}
\usepackage{tikz}
\usetikzlibrary{decorations}
\usetikzlibrary{decorations.pathreplacing}
\tikzset{individu/.style={draw,thick}}

\theoremstyle{plain}
\newtheorem{theorem}{Theorem}[section]
\newtheorem{corollary}[theorem]{Corollary}
\newtheorem{lemma}[theorem]{Lemma}
\newtheorem{proposition}[theorem]{Proposition}

\theoremstyle{definition}
\newtheorem{definition}[theorem]{Definition}

\theoremstyle{remark}
\newtheorem{remark}[theorem]{Remark}

\numberwithin{equation}{section}

\newcommand{\N}{\mathbb{N}}
\newcommand{\Z}{\mathbb{Z}}
\newcommand{\R}{\mathbb{R}}
\newcommand{\C}{\mathbb{C}}

\newcommand{\calC}{\mathcal{C}}

\newcommand{\U}{\mathbb{U}}

\newcommand{\T}{\mathbf{T}}

\newcommand{\ind}[1]{\mathbf{1}_{\left\{#1\right\}}}

\newcommand{\ceil}[1]{{\left\lceil #1 \right\rceil}}

\newcommand{\crochet}[1]{{\langle #1 \rangle}}
\renewcommand{\bar}[1]{\overline{#1}}
\renewcommand{\tilde}[1]{\widetilde{#1}}
\renewcommand{\hat}[1]{\widehat{#1}}\newcommand{\ex}{\mathbf{e}}
\newcommand{\e}{\mathrm{e}}
\newcommand{\dd}{\mathrm{d}}
\newcommand{\egaldistr}{{\overset{(d)}{=}}}

\DeclareMathOperator{\E}{\mathbb{E}}

\renewcommand{\P}{\mathbb{P}}

\newcommand{\calF}{\mathcal{F}}

\newcommand{\x}{\mathbf{x}}

\renewcommand{\rho}{\varrho}
\renewcommand{\epsilon}{\varepsilon}

\title{Infinitely ramified point measures\\ and branching L\'evy processes}
\author{Jean Bertoin\footnote{Institut f\"ur Mathematik, Universit\"at Z\"urich, Switzerland.}  \and Bastien Mallein${}^*$}
\date{\today}

\begin{document}

\maketitle

\begin{abstract}
We call a random point measure \textit{infinitely ramified} if for every $n\in \N$, it has the same distribution as the $n$-th generation of some branching random walk. On the other hand, \textit{branching L\'evy processes} model the evolution of a population in continuous time, such that individuals move in space independently, according to some L\'evy process, and further beget progenies according to some Poissonian dynamics, possibly on an everywhere dense set of times. Our main result connects these two classes of processes much in the same way as in the case of infinitely divisible distributions and L\'evy processes: the value at time $1$ of a branching L\'evy process is an infinitely ramified point measure, and conversely, any infinitely ramified point measure can be obtained as the value at time $1$ of some branching L\'evy process.
\end{abstract}

\section{Introduction}
The classical works of L\'evy, Khintchin, Kolmogorov and It\^o have unveiled the fine structure of infinitely divisible distributions on $\R^d$, and their connections with processes with independent and stationary increments. In short, the purpose of this work is to develop an analogous theory in the setting of point measures and branching processes. Let us start by recalling some of the well-known connections between infinitely divisible laws, random walks and L\'evy processes. The formulation is tailored to fit our purposes; we also refer to 
 \cite{Sato} for a textbook treatment of this topic.
 
For the sake of simplicity, we focus on the dimension $d=1$ and work with $[-\infty, \infty)$ as state space, where the boundary point $-\infty$ serves as cemetery state. We call a discrete time process $(S_n: n\in \Z_+)$ with values in $[-\infty, \infty)$ and started from $S_0=0$ a {\em random walk}, if for every integers $n,k\geq 0$, one can express $S_{n+k}$ in the form $S_{n+k}=S_n+S'_k$, where $S'_k$ has the same law as $S_k$ and is further independent of $\sigma(S_i: 0\leq i \leq n)$. A distribution on $[-\infty, \infty)$, say $\rho$, is then called {\em infinitely divisible} if for every $n\in\N$, there is a random walk such that $S_n$ has the law $\rho$.

This naturally leads us to consider random walks indexed by dyadic rational times. Specifically, we write
\[ D=\left\{k 2^{-n}: n,k\in\Z_+\right\} \]
for the set of dyadic rational times and consider a process $(\xi_t: t\in D)$ such that for every $s,t\in D$, one has $\xi_{t+s}=\xi_t+\xi'_s$, where $\xi'_s$ has the same law as $\xi_s$ and is independent of $\sigma(\xi_r: r\in D, r \leq t)$. Equivalently, the \textit{discrete time skeletons} (i.e. the processes $(\xi_{k2^{-n}}:k \in \Z_+)$ for all integers $n$) are random walks. Excluding implicitly the degenerate case when $\xi_1= -\infty$ a.s., $\xi_1$ then has an infinitely divisible distribution, and conversely, any infinitely divisible distribution on $[-\infty, \infty)$ which is not degenerate (i.e. not the Dirac mass at $-\infty$) can be obtained in this setting.
Further, one can extend the process $\xi$ to non-negative real times and get a process $(\xi_t: t\in \R_+)$ with c\`adl\`ag paths a.s. The latter is a {\em L\'evy process}, possibly killed at some constant rate, and its structure is described by the celebrated L\'evy-It\^o decomposition. This identifies the continuous part of $\xi$ as a Brownian motion with constant drift and its jumps as a Poisson point process whose intensity is determined by the so-called L\'evy measure of $\xi$.

We next turn our attention to random point measures, and first introduce some notation in this setting. We write $\mathcal{P}$ for the space of point measures on~$\R$ that assign a finite mass to semi-infinite intervals. Specifically, $\mu\in{\mathcal P}$ if and only if $\mu$ is a measure on $\R$ with integer-valued tail:
\[ \bar{\mu}(x)\coloneqq \mu((x,\infty))\in \Z_+\quad \text{ for every }x\in \R. \]
Repeating the atoms of $\mu$ according to their multiplicity and ranking them in the non-increasing order yields a sequence $\x=(x_n: n\in\N)$, called the \textit{ranked sequence of atoms} of $\mu$, where we set $x_n=-\infty$ for $n>\mu(\R)$. 
 We shall therefore often identify $\mu$ with the ranked sequence of its atoms, and thus $\mathcal{P}$ with the space of non-increasing sequences $\x$ in $[-\infty, \infty)$ with $\lim_{n\to \infty}x_n=-\infty$. That is, we shall use indifferently the notation $\mu$ or $\x$, in the sense that then
\[ \mu = \sum_{n=1}^{\infty} \delta_{x_n}, \]
with the convention that the possible atoms at $-\infty$ are discarded, i.e. $\delta_{-\infty}=0$. In particular, the zero measure is identified with the sequence $\varnothing\coloneqq (-\infty, \ldots)$. We further denote by $\tau$ the translation operator on ${\mathcal P}$, setting \[ \tau_y \x=\x+y=(x_n+y: n\in\N)\]
for every $y\in[-\infty, \infty)$, and equivalently
\[ \tau_y \mu = \sum_{n \geq 1} \delta_{x_n + y}. \]
 Observe that
for $y=-\infty$, by our convention, $\tau_{-\infty} \mu = \varnothing$ for any $\mu\in{\mathcal P}$.

A {\em branching random walk} is a process in discrete time $(Z_n: n\in \Z_+)$ with values in $\mathcal{P}$ and $Z_0=\delta_0$ a.s., such that for every $n,k\geq 0$, the point measure $Z_{n+k}$ can be expressed in the form 
\[ Z_{n+k} = \sum_{i=1}^{\infty} \tau_{x_i} Z^i_k,\quad \text{ where }\x=Z_n, \]
with $(Z^i_k: i\in\N)$ i.i.d. copies of $Z_k$ which are independent of $\sigma(Z_j: 0\leq j \leq n)$. This condition is  referred to  as the (simple) branching property. 

In order to ensure that the number of particles in semi-infinite intervals never explodes in finite time, that is $Z_n \in \mathcal{P}$ a.s. for all $n\geq 1$,  one usually  further requests non-degeneracy of the Laplace transform of the intensity. Namely, one assumes that
\begin{equation} \label{eqn:integrabilityInfinitelyRamified} 
\text{there exists 
$\theta \geq 0$ such that $0<\E(\crochet{Z_1,\ex_\theta})<\infty$,}
\end{equation}
with  the notation 
\[\crochet{\mu,f} = \int_{\R} f \dd \mu = \sum_{i \in \N} f(x_i)\quad \text{ and } \quad \ex_\theta: y \in \R \mapsto \e^{\theta y}, \]
and the convention that $f(-\infty)=0$. The requirement that $\theta\geq 0$ in \eqref{eqn:integrabilityInfinitelyRamified} is of course just a matter of convenience, since the case $\theta<0$ follows  by reflexion with obvious modifications. There exist non-exploding branching random walks for which \eqref{eqn:integrabilityInfinitelyRamified} fails; however most works in that field rely on such non-degeneracy assumption. The condition $\E(\crochet{Z_1,\ex_\theta})>0$ is equivalent to $Z_1 \neq \varnothing$ with positive probability, which is simply a non-degeneracy assumption.

Our first object of interest in the present work is the family of {\em infinitely ramified point measures}, formed by the 
random point measures ${\mathcal Z}$ which have the property that for every $n\in\N$, 
${\mathcal Z}$ has the same distribution as 
the $n$-th generation of some branching random walk. Throughout  this article, we restrict our attention to random point measures satisfying \eqref{eqn:integrabilityInfinitelyRamified}.

Our second object of interest is the family of {\it branching L\'evy processes} that we introduce only informally here, postponing the rigorous construction to Section~\ref{sec:branchingLevyProcess}. Loosely speaking, a branching L\'evy process is a particle system in real time, starting from a single particle located at $0$, where particles evolve in $[-\infty, \infty)$ independently and according to some (possibly killed) L\'evy process. They further beget children in a Poissonian manner, where the location of birth of each child is given by some random shift of the location of its parent at the time of the birth event. Branching Brownian motions with constant drift (see, for instance, Chapter 5 in Bovier \cite{Bovier}) form the class of branching L\'evy processes with continuous ancestral trajectories. Some fairly general instances of branching L\'evy processes with discontinuous ancestral trajectories have appeared the study of so-called homogeneous fragmentations \cite{Be-Rou} and compensated fragmentation \cite{BeCF}.

Similarly to Lévy processes, the distribution of a branching L\'evy process is characterized by a triple $(\sigma^2, a, \Lambda)$, where
\begin{itemize}
 \item $\sigma^2\geq 0$ is the variance of the Brownian component of the motion of typical individuals;
 \item $a \in \R$ is the drift coefficient of that motion;
 \item $\Lambda$ is a sigma-finite measure on ${\mathcal P}$, referred to as the L\'evy measure (of the branching L\'evy process).
\end{itemize}
Further, the L\'evy measure $\Lambda$ has to satisfy certain requirements that are better understood if we view a ranked sequence $\x=(x_n: n\in\N)$ in ${\mathcal P}$ as a pair $\x=(x_1, \x_2)$, where $\x_2\coloneqq (x_{n+1}: n\in\N)$. 
Recall also that $\varnothing=(-\infty, \ldots)$, so $(0,\varnothing)=(0, -\infty, \dots)$ is the ranked sequence of atoms of the Dirac point mass at $0$. 
The first requirement for $\Lambda$ to be a L\'evy measure is
\begin{equation}
 \label{eqn:integrabilityPi3}
 \Lambda(\{(0,\varnothing)\})=0 \quad\hbox{and} \quad \int_{\mathcal P}(1\wedge x_1^2)\Lambda(\dd \x)<\infty .
\end{equation}
Next, for some parameter $\theta\geq 0$, 
$\Lambda$ has to fulfil a pair of integrability conditions, namely
\begin{equation}
 \label{eqn:integrabilityPi}
 \int_{\mathcal P} \ind{x_1>1} \e^{\theta x_1} \Lambda(\dd \x) < \infty, 
\end{equation}
and
\begin{equation}
 \label{eqn:integrabilityPi2}
 \int_{\mathcal P} \sum_{k=2}^{\infty} \e^{\theta x_k} \Lambda(\dd \x) < \infty,
\end{equation}
with the convention that $ \e^{\theta x}=0$ when $x=-\infty$.

The first requirement enables us to view the image measure $\Lambda_1$ of $\Lambda$ by the first projection $\x \mapsto x_1$ as the L\'evy measure of a L\'evy process. We stress that $\Lambda_1$ may have an atom at $-\infty$: $\Lambda_1(\{-\infty\})=\Lambda(\{\varnothing\})$ is always finite by \eqref{eqn:integrabilityPi3}, but may be strictly positive, and then should be viewed as a pure death rate. The possibly killed L\'evy process that governs the motion of particles has Gaussian coefficient $\sigma^2$ and L\'evy measure $\Lambda_1$. Note also from \eqref{eqn:integrabilityPi} that this L\'evy process has a finite exponential moment of order $\theta$. In turn, the image measure $\Lambda_2$ of~$\Lambda$ by the second projection $\x\mapsto \x_2$ and restricted to the space of non-zero point measures, describes the intensity of the relative positions at birth of the newborn children. We will prove later on that assumptions \eqref{eqn:integrabilityPi} and \eqref{eqn:integrabilityPi2} are in fact equivalent to the requirement that the one-dimensional marginals of the branching Lévy process satisfy the condition \eqref{eqn:integrabilityInfinitelyRamified}.

More precisely, given a Brownian motion $B$ and an independent Poisson point process $N$ on $\R_+ \times \mathcal{P}$ with intensity $\dd t \Lambda(\dd \x)$, the initial individual in the branching L\'evy process moves as a L\'evy process with Brownian component $B$. For each atom $(t,\x)$ of $N$, this individual makes a jump of size $x_1$ at time $t$, and simultaneously produces\footnote
{We stress that, alike Crump-Mode-Jagers generalized branching processes, and contrary to descriptions of Galton-Watson type, individuals do not die at the times when they beget children.} offspring around its pre-jump position according to the sequence $\x_2$. 
 The full trajectory of the ancestor is a L\'evy process with characteristics $(\sigma^2,a,\Lambda_1)$.

We say that a branching L\'evy process has {\em finite birth intensity} if \eqref{eqn:integrabilityPi2} holds with $\theta=0$. In that case, this process can be seen as a special case of a branching Markov process in continuous time, as introduced by Hering \cite{Hering}: each individual moves independently according to a L\'evy process, and at independent exponential times, they are replaced by a family of children positioned around their parent according to an i.i.d. copy of a point measure.
Roughly speaking, branching L\'evy processes with infinite birth intensity can be constructed as the increasing limit of a sequence of branching L\'evy processes with finite birth intensity; the condition \eqref{eqn:integrabilityPi2} ensures that no explosion in finite time occurs. 

Our main result can be stated as follows. 
\begin{theorem}
\label{thm:main} \begin{enumerate}
\item [(i)]
Let $\mathcal{Z}$ be an infinitely ramified point measure. Assume that 
\begin{equation} \label{eqn:integrabilityIR} 
0<\E(\crochet{\mathcal{Z}_1,\ex_\theta})<\infty
\end{equation}
holds for some $\theta \geq0$.
Then there exists a branching Lévy process $(Z_t : t\geq 0)$ with 
$$\mathcal{Z}\, \egaldistr \,Z_1.$$

\item [(ii)] Reciprocally, let $(Z_t : t \geq 0)$ be a branching Lévy process with characteristics $(\sigma^2,a,\Lambda)$ satisfying \eqref{eqn:integrabilityPi3}, \eqref{eqn:integrabilityPi} and \eqref{eqn:integrabilityPi2}. Then $Z_1$ is an infinitely ramified point measure that fulfills \eqref{eqn:integrabilityInfinitelyRamified}. Further, for every $t\geq 0$ and $z\in\C$ with $\Re z=\theta$, one has
$$\E(\crochet{Z_t,\ex_z})= \exp(t\kappa(z)),$$
where 
$$\kappa(z)\coloneqq \frac{\sigma^2}{2} z^2 + az + \int_{{\mathcal P}} \left(\sum_{j=1}^{\infty} \e^{z x_j}-1-z x_1{\mathbf 1}_{|x_1|<1}\right) \Lambda(\dd \x).$$
\end{enumerate}
\end{theorem}

The proof of Theorem \ref{thm:main}(i) relies on the construction of an intermediary process in dyadic rational times,  $(Z_t:t \in D)$, called a \textit{nested branching random walk}, meaning that each discrete-time skeleton is a branching random walk. We prove that to each infinitely ramified point measure is associated a nested branching random walk such that $\mathcal{Z}\, \egaldistr \,Z_1$, and that for each nested branching random walk, there exists a unique triplet $(\sigma^2,a,\Lambda)$ such that the restriction of a branching Lévy process with these characteristics to dyadic rational times has the law of that nested branching random walk.

There are two aspects of Theorem \ref{thm:main} that may seem unsatisfactory when one compares with the  classical analog for infinitely distributions and L\'evy processes. First, this result holds under the non-degeneracy assumption \eqref{eqn:integrabilityIR}; we have already argued that this is however a very natural and common hypothesis in the framework of branching random walks. Second, we have been unable to establish uniqueness of the distribution of branching Lévy processes associated to an infinitely ramified point measure.  In the classical framework, it is known that the characteristic function of  an infinite divisible law is never $0$ and thus possesses a unique continuous complex $n$-th root, which is then the characteristic function of the unique $n$-th root  of that law in the sense of the convolution operation.  Unfortunately, this argument cannot be transferred to the framework of point measures. We were only able to prove the existence of an $n$-th root\footnote{In the sense induced by the branching operation, that is, an $n$-th root of $\mathcal{Z}$ is  the reproduction law of a branching random walk $(Z_n: n\geq 0)$ such that $Z_n \, \egaldistr \, \mathcal{Z}$.} of an infinitely ramified point measure $ \mathcal{Z}$ using a compactness argument that relies again crucially on \eqref{eqn:integrabilityIR}. The lack of a handy characterization of such an $n$-th root in terms of the law of $\mathcal{Z}$ hindered us from tackling the issue of uniqueness. We stress that nonetheless, uniqueness of the characteristic triplet $(\sigma^2,a, \Lambda)$ of a branching Lévy process holds, see the forthcoming Remark~\ref{rem:uniqueness}.

\paragraph*{Organisation of the paper.}
The existence of an embedding of an infinitely ramified point measure into a nested branching random walk is proven in Section~\ref{sec:infinitelyRamifiedPointMeasure}. Roughly speaking, the key issue is to establish that every infinitely ramified point measure has the same law as the second generation of a branching random walk whose reproduction law is also given by an infinitely ramified point measure. This will be achieved by compactness arguments on the space of probability distributions on ${\mathcal P}$. 

In Section \ref{sec:infinitelyDivisibleFirstProperties}, 
we shall start by proving that a nested branching random walk $Z$ always possesses an a.s. c\`adl\`ag extension in real time. We also establish a many-to-one formula for one-dimensional distributions in this framework, and extend the simple branching property to stopping times. 

Section~\ref{sec:finiteBirthIntensity} is devoted to nested branching random walks with finite birth intensity. Analyzing the first branching time naturally yields the notion of branching L\'evy process in this simple case, and Theorem~\ref{thm:main} is then easily checked in this setting. 

General branching L\'evy processes are constructed in Section~\ref{sec:branchingLevyProcess} as increasing limits of branching L\'evy processes with finite birth intensity, adapting arguments in \cite{BeCF}. One readily observes that these processes satisfy the branching property as well as \eqref{eqn:integrabilityInfinitelyRamified}, which establishes Theorem \ref{thm:main} (ii). 

Our main task in Section~\ref{sec:genealogicalStructure} is to equip nested branching random walks with a natural genealogy. This enables us to define ancestral lineages and establish a pathwise version of the many-to-one formula. This further allows us to introduce a censoring procedure, by killing certain individuals depending on the behavior of their ancestral lineage. The upshot is that this yields an approximation of a nested branching random walk by a sequence of nested branching random walks with finite birth intensities, which in turn enables us to complete the proof of Theorem~\ref{thm:main}.

\section{From infinitely ramified point measures to \texorpdfstring{\\}{} nested branching random walks}
\label{sec:infinitelyRamifiedPointMeasure}

A \emph{nested branching random walk} is a $\mathcal{P}$-valued process $(Z_t:t \in D)$ with $Z_0 = \delta_0$ that satisfies the simple branching property:
\begin{description}
 \item[(B)] For every $s,t \in D$, we have
$$Z_{t+s} = \sum_{n=1}^{\infty} \tau_{x_n} Z^n_t,\quad \text{ where }\x=Z_s,$$
with $(Z^n_t: n\in\N)$ a family of i.i.d. copies of $Z_t$, independent of $\calF_s$, 
\end{description}
where $\calF_s \coloneqq \sigma(Z_r : r\in D, r \leq s)$ denotes the natural filtration of $Z$.
The terminology refers to the fact that (B) is equivalent to the requirement  that the discrete time skeletons $(Z_{k2^{-n}}: k\in \Z_+)$ of $Z$ are branching random walks. 

The main purpose of this section is to establish the following embedding property of infinitely ramified point measures into nested branching random walks. 
We will then prove that nested branching random walks possess a unique c\`adl\`ag extension to real times $(Z_t: t\in \R_+)$, and that the branching property holds more generally at stopping times.

\begin{proposition}
\label{thm:irpmToNbrw}
Let $\mathcal{Z}$ be an infinitely ramified point measure satisfying \eqref{eqn:integrabilityIR}. There exists a nested branching random walk $(Z_t : t \in D)$ with 
$$\mathcal{Z}\, \egaldistr \,Z_1.$$
\end{proposition}

Let $\mathcal{Z}$ be an infinitely ramified point measure (recall that we always assume that \eqref{eqn:integrabilityIR} holds). In words, for every $n \in \N$ there exists a branching random walk such that the $n$-th generation of that process has same law as $\mathcal{Z}$. In this section, we fix $\theta \geq 0$ such that \eqref{eqn:integrabilityIR} is fulfilled, so
\[\kappa(\theta)\coloneqq \ln \E\left(\crochet {\mathcal{Z}, \ex_{\theta}}\right)\in\R.\]
We introduce the $\theta$-exponentially weighted intensity of $\mathcal{Z}$, say ${m}_{\theta}$, which is the probability measure on $\R$ defined by 
$$\crochet{{m}_{\theta}, f}\coloneqq \e^{-\kappa(\theta)}\E\left(\crochet{\mathcal{Z},\ex_{\theta}f}\right),$$
where $f\in L^{\infty}(\R)$ denotes a generic bounded measurable function. 

One of the key tools for the study of branching random walks is the well-known many-to-one formula, that can be traced back at least to the early work of Kahane and Peyri\`ere \cite{KaP,Pey}; see Theorem~1.1 in \cite{ShiSF}. This result enables us to identify $m_{\theta}$ as the distribution of a real-valued random walk evaluated after $n$ steps. Hence, the assumption that $\mathcal{Z}$ is an infinitely ramified point measure implies that ${m}_{\theta}$ is infinitely divisible.
We write $\Psi$ for its characteristic exponent, that is $\Psi: \R\to \C$ is the unique continuous function with $\Psi(0)=0$ such that 
\begin{equation}\label{eq:psi}
\crochet{{m}_{\theta}, \ex_{i {r}}}= \e^{-\kappa(\theta)}\E\left(\crochet{\mathcal{Z},\ex_{\theta+i {r}}}\right)=
\e^{\Psi({r})}\,, \qquad {r} \in\R.
\end{equation}
This enables us to introduce a L\'evy process (without killing) $\xi=(\xi_t: t\in\R_+)$ with characteristic exponent $\Psi$, i.e. satisfying $\E(\e^{i{r} \xi_t})=\e^{t\Psi({r})}$ for all ${r} \in \R$ and $t \geq 0$. Note that the law of $\xi$ is determined by the law of $\mathcal{Z}$. In the sequel, it will be convenient to use the notation
$$\kappa(\theta+i{r})\coloneqq \kappa(\theta)+\Psi({r}) \qquad \text{for every }{r} \in \R,$$
and refer to $\kappa$ as the {\em cumulant} of $\mathcal{Z}$.

A special case of the many-to-one formula for ${\mathcal Z}$, which will be useful in this section, can be stated as follows. 

\begin{lemma}[Many-to-one formula]
\label{lem:manytoonePrimal}
For any $n \in \N$, if $(Z^{(n)}_k: k\in \Z_+)$ is a branching random walk such that $Z^{(n)}_n$ has same law as $\mathcal{Z}$, then for all measurable functions $f : \R \to \R_+$ and all integers $j\in\N$, we have
\[
 \E\left( \crochet{Z^{(n)}_j,f} \right) = \e^{\kappa(\theta)j/n}\E\left( \e^{-\theta \xi_{j/n}} f(\xi_{j/n}) \right).
\]
\end{lemma}

\begin{proof}
Applying the many-to-one formula to the branching random walk $Z^{(n)}$, there exists a random walk $S^{(n)}$ such that
\[
 \E\left( \crochet{Z^{(n)}_n,\ex_\theta f} \right) = \e^{\kappa(\theta)}\E\left( f(S^{(n)}_n) \right),
\]
for all measurable positive functions $f$. Therefore, $S^{(n)}_n$ has the same law as $\xi_1$, from which we conclude that $S^{(n)}_j$ has same law as $\xi_{j/n}$. Using again the many-to-one formula, we have
\[
 \E\left( \crochet{Z^{(n)}_j, f} \right) = \e^{\kappa(\theta)j/n}\E\left(\e^{-\theta \xi_{j/n}} f(\xi_{j/n}) \right)
\]
for all non-negative measurable functions $f$, concluding the proof.
\end{proof}

The rest of this section is devoted to establish Proposition \ref{thm:irpmToNbrw}. The proof relying on compactness arguments, we carefully introduce the topological spaces we will be using. To start with, let 
\[{\mathcal P}_{\theta}\coloneqq\{\mu\in {\mathcal P}: \crochet{\mu,\ex_{\theta}}<\infty\}\]
be the subspace of point measures $\mu$ such that $\ex_{\theta}\mu$ is a finite measure; we henceforth view $Z$ as a process with values in ${\mathcal P}_{\theta}$. The space of finite measures on $\R$ is naturally endowed with the topology of weak convergence, and the set $\{\ex_{\theta}\mu: \mu\in{\mathcal P}_{\theta}\}$ is a closed subset thereof. We thus say that a sequence $(\mu_n: n\in\N)$ in ${\mathcal P}_{\theta}$ converges to $\mu\in {\mathcal P}_{\theta}$ and write
 $$\lim_{n\to\infty }\mu_n=\mu\quad\hbox{ in }{\mathcal P}_{\theta}$$
if and only if 
\[
 \forall f \in \calC_b, \lim_{n \to \infty} \crochet{\mu_n, \ex_\theta f} = \crochet{\mu,\ex_\theta f},
\]
where $\mathcal{C}_b$ is the set of continuous bounded functions on $\R$. Plainly, convergence in ${\mathcal P}_{\theta}$ is stronger than vague convergence, or convergence of tail functions (pointwise, except possibly at discontinuity points of the limit). In turn, the latter is also equivalent to the simple convergence of the ranked sequence of the atoms. As a point measure $\mu \in \mathcal{P}_\theta$ can be identified with the finite measure $\ex_\theta. \mu$ on $\R$, we note that $\mathcal{P}_\theta$ can be seen as a closed subspace of the set of finite measures on $\R$ endowed with the topology of the weak convergence. Therefore, the space $\mathcal{P}_\theta$ endowed with this topology is a Polish space (see, for instance, Lemma 4.5 in Kallenberg \cite{Kall}). 

We write $\mathbf{P}_\theta$ for the set of probability measures on $\mathcal{P}_\theta$, which is also endowed with the topology of the weak convergence. We give a simple condition for a subset of $\mathbf{P}_\theta$ to be compact. We call that a continuous function $f: \R\to (0,\infty)$ norm-like if
\[
 \lim_{x \to \infty} f(x) = \lim_{x \to -\infty} f(x) = \infty,
\]
and observe that any family ${\mathbf F}$ of finite measures on $\R$ is tight if and only if there exists a norm-type function $f$ with $\sup_{m\in {\mathbf F}}\crochet{m,f}<\infty$; see e.g. Lemma D 5.3 in Meyn and Tweedie \cite{MeynTwee}. 
\begin{lemma}
\label{lem:compact}
Let $\mathbf{K}$ be a closed non-empty subset of $\mathbf{P}_\theta$. If there exists a continuous norm-like  function $f$ satisfying
\[
 \sup_{P \in \mathbf{K}} \int \crochet{\nu,\ex_\theta f} P(\dd \nu) <\infty,
\]
then $\mathbf{K}$ is compact.
\end{lemma}

\begin{proof}
Up to multiplying $f$ by a constant, we may assume that $f \geq 1$. Let $\epsilon \in (0,1)$, and for any $n \in \N$, set
\[
 K_{n,\epsilon} \coloneqq \left\{ x \in \R : f(x) \leq n2^n/\epsilon \right\}.
\]
which is a compact subset of $\R$, and
\[
 {\mathcal L}_\epsilon = \left\{ \nu \in \mathcal{P}_\theta : \crochet{\nu,\ex_\theta} \leq 1/\epsilon, \enskip \forall n \in \N, \crochet{\nu,\ex_\theta \mathbf{1}_{K_{n,\epsilon}^c}} \leq 1/n \right\}.
\]
By definition, ${\mathcal L}_\epsilon$ is tight and obviously closed. Hence by Prokhorov's theorem, this is a compact subset of $\mathcal{P}_\theta$. Moreover, for any $P \in \mathbf{K}$, we have
\begin{multline*}
\qquad P({\mathcal L}_\epsilon^c) \leq P\left( \left\{ \nu \in \mathcal{P}_\theta : \crochet{\nu,\ex_\theta} > 1/\epsilon\right\} \right)\\
+\sum_{n =1}^\infty P\left( \left\{ \nu \in \mathcal{P}_\theta : \crochet{\nu,\ex_\theta \mathbf{1}_{K_{n,\epsilon}^c}} > 1/n \right\}\right). \qquad
\end{multline*}
By Markov inequality, 
$$P\left( \left\{ \nu \in \mathcal{P}_\theta : \crochet{\nu,\ex_\theta} > 1/\epsilon \right\} \right)\leq \epsilon \int \crochet{\nu,\ex_\theta f} P(\dd \nu),$$ and
\begin{eqnarray*}
 \sum_{n =1}^\infty P\left( \left\{ \nu \in \mathcal{P}_\theta : \crochet{\nu,\ex_\theta \mathbf{1}_{K_{n,\epsilon}^c}} > 1/n \right\}\right)
 &\leq & \sum_{n=1}^\infty n\frac{\epsilon}{n2^n}\int_{\mathcal{P}_\theta} \crochet{\nu,\ex_\theta f} P(\dd \nu)\\
 &\leq& \epsilon \sup_{Q \in \mathbf{K}} \int \crochet{\nu,\ex_\theta f} Q(\dd \nu),
\end{eqnarray*}
which enables us to conclude the proof using again Prokhorov's theorem.
\end{proof}

We next introduce a convolution-type operation on $\mathbf{P}_\theta$, related to the dynamics of branching random walks. By analogy with the convolution operation associated to the random walk, for any pair $P,Q \in \mathbf{P}_\theta$, we denote by $P \circledast Q$ the distribution of the first generation of 
a branching random walk with reproduction law $Q$ and started from a random point measure distributed according to $P$. In other words, writing $\mu$ for a random point measure with law $P$, $\x$ for its ranked sequence of atoms, and $(\nu_j:j\in \N)$ for i.i.d. random point measures with law $Q$, then $P \circledast Q$ is the law of the random measure $\sum_{j \in \N} \tau_{x_j} \nu_j$. By a straightforward computation, there is the identity
\begin{equation}
 \label{eqn:multiplicatif}
 \int \crochet{\mu,\ex_\theta} P \circledast Q(\dd \mu) = \int \crochet{\mu,\ex_\theta} P (\dd \mu) \times \int \crochet{\nu,\ex_\theta} Q (\dd \nu),
\end{equation}
which ensures that $P \circledast Q \in \mathbf{P}_\theta$ for any $P,Q \in \mathbf{P}_\theta$ such that the right-hand side in \eqref{eqn:multiplicatif} is finite. 
We now study the regularity of this operator.

\begin{lemma}
\label{lem:continue++}
Let $(P_n: n \in \N)$ and $(Q_n:n \in \N)$ be two sequences in $\mathbf{P}_\theta$ such that
\[
 \lim_{n \to \infty} P_n = P \quad \text{and} \quad \lim_{n \to \infty} Q_n = Q \quad \text{in } \mathbf{P}_\theta.
\]
If 
\begin{equation}\label{eqn:cvEnManytoone}
\sup_{n\in \N} \int \crochet{\nu,\ex_\theta } Q_n(\dd \nu) <\infty, 
\end{equation}
and further, there exists a continuous norm-like function $g$ such that
\begin{equation}
 \sup_{n \in \N} \int \crochet{\mu,\ex_\theta g} P_n(\dd \mu) < \infty , \label{eqn:cvEnManytoone2}
\end{equation}
then
$$\lim_{n \to \infty} P_n \circledast Q_n = P \circledast Q\qquad \text{in }{\mathbf P}_{\theta}.$$
\end{lemma}

\begin{proof}
By Skorohod's representation theorem, we construct random point measures $\mu_n,\mu, \nu_n$ and $\nu$ in $\mathcal{P}_\theta$, with laws $P_n,P,Q_n$ and $Q$ respectively, and such that
\[
 \lim_{n \to \infty} \mu_n = \mu \quad \text{and} \quad \lim_{n \to \infty} \nu_n = \nu \quad \text{ a.s. in } \mathcal{P}_\theta.
\]
We introduce a sequence $\left(\left( \left(\nu^j_n\right)_{n\in\N},\nu^j\right) :j \in \N\right)$ of i.i.d. copies of $\left((\nu_n)_{n\in\N},\nu\right)$ that are further independent of $\left((\mu_n)_{n\in\N},\mu\right )$, and write
\[
 \rho_n = \sum_{j \in \N} \tau_{x^n_j} \nu^j_n \quad \text{and} \quad \rho = \sum_{j \in \N} \tau_{x_j} \nu^j,
\]
where $\x^n=(x^n_j: j\in\N)$ and $\x=(x_j: j\in \N)$ are the ranked sequence of atoms of $\mu_n$ and $\mu$, respectively. Then $\rho_n$ has law $P_n \circledast Q_n$, $\rho$ has law $P \circledast Q$, and we aim at proving that $\lim_{n\to \infty} \rho_n=\rho$ in probability in ${\mathcal P}_{\theta}$. That is, by an argument of separability, that for every function $f \in \mathcal{C}_b$,
\begin{equation} \label{eqn:goal}
\lim_{n\to \infty} \crochet {\rho_n,\ex_{\theta} f}= \crochet {\rho, \ex_{\theta}f} \qquad \text{in probability.}
\end{equation}
Without loss of generality, we may focus henceforth on the case $0\leq f \leq 1$. 

To start with, recall that convergence in $\mathcal{P}_\theta$ is stronger than pointwise convergence of the ranked sequences of atoms. Hence, for any fixed $k \in \N$ and any $f \in \mathcal{C}_b$, we have
\begin{equation}\label{eqn:basic}
 \lim_{n \to \infty} \crochet{ \tau_{x^n_k} \nu^n_k,\ex_\theta f}=\crochet{ \tau_{x_k} \nu_k,\ex_\theta f} \quad \text{a.s.}
\end{equation}
Next, fix a large integer $N$ and for every $n\in\N$ and $a>0$, introduce the event
$\Lambda_n(N,a)\coloneqq \{x^n_N< -a\}$. Plainly, on that event, 
 there is the inequality
$$0\leq \crochet{\rho_n,\ex_\theta f}-
\sum_{k=1}^N \crochet{ \tau_{x^n_k} \nu^n_k,\ex_\theta f} \leq \sum_{k=1}^{\infty} \crochet{ \tau_{x^n_k} \nu^n_k,\ex_\theta }{\mathbf 1}_{\{x^n_k<-a\}}.$$
By \eqref{eqn:multiplicatif}, the expectation of the right-hand side equals
$$
\int \crochet{\mu,\ex_\theta {\mathbf 1}_{(-\infty, -a)}} P_n(\dd \mu)\times 
\int \crochet{\nu,\ex_\theta} Q_n(\dd \nu).$$

By the elementary inequality, 
$$\crochet{\mu,\ex_\theta {\mathbf 1}_{(-\infty, -a)}}\leq \crochet{\mu,\ex_\theta g}/\inf\{g(x): x<-a\},$$
we see that \eqref{eqn:cvEnManytoone}, \eqref{eqn:cvEnManytoone2}
 and the fact that $g$ is norm-like ensure that for every $\varepsilon >0$, we can choose $a>0$ sufficiently large such that
 $$\E\left( \sum_{k=1}^{\infty} \crochet{ \tau_{x^n_k} \nu^n_k,\ex_\theta f}{\mathbf 1}_{\{x^n_k<-a\}}\right)\leq \varepsilon ^2\qquad \text{for all }n\in\N.$$
 {\it A fortiori}, by the Markov inequality, we have
 $$\P\left( \sum_{k=1}^{\infty} \crochet{ \tau_{x^n_k} \nu^n_k,\ex_\theta f}{\mathbf 1}_{\{x^n_k<-a\}}>\varepsilon \right)\leq \varepsilon \qquad \text{for all }n\in\N.$$

We next bound the probability of the complementary event $\Lambda_n(N,a)^c$ using again Markov's inequality: 
$$\P(x^n_N\geq -a)=\P(\mu_n([-a,\infty))\geq N)\leq \frac{\e^{\theta a}}{N} \E(\crochet{\mu_n,\ex_{\theta}}).$$
So by \eqref{eqn:cvEnManytoone2}, we may choose $N$ large enough so that
$\P(\Lambda_n(N,a))\geq 1-\varepsilon$ for all $n\in\N$.
To summarize, 
we have shown that for every $\varepsilon>0$, we can choose $N$ sufficiently large so that for all $n\in\N$
\[
\P\left( \crochet{\rho_n,\ex_\theta f} - \sum_{j=1}^N \crochet{ \tau_{x^n_j} \nu^n_j,\ex_\theta f} > \epsilon\right) \leq 2\varepsilon. 
\]

Since \eqref{eqn:basic} entails that for each fixed $N$, 
$$\lim_{n\to \infty} \sum_{j=1}^N \crochet{ \tau_{x^n_j} \nu^n_j,\ex_\theta f}= \sum_{j=1}^N \crochet{ \tau_{x_j} \nu_j,\ex_\theta f} \qquad \text{ a.s.} $$
and plainly
$$\lim_{N\to \infty} \sum_{j=1}^N \crochet{ \tau_{x_j} \nu_j,\ex_\theta f}=\crochet{\rho, \ex_\theta f} \qquad \text{ a.s.}, $$
we conclude that \eqref{eqn:goal} holds, which completes the proof.
\end{proof}

We are now able to establish Proposition \ref{thm:irpmToNbrw}.

\begin{proof} [Proof of Proposition \ref{thm:irpmToNbrw}]
We denote by $P\in \mathbf{P}_{\theta}$ the law of the infinitely ramified point measure $\mathcal{Z}$,
and introduce the space of $\circledast$-roots of $P$, viz.
\[
\mathbf{R}(P) = \left\{ Q \in \mathbf{P}_\theta : Q\circledast Q = P \right\}. 
\]
In words, $Q\in \mathbf{R}(P)$ if and only if the second generation of a branching random walk with reproduction law $Q$ has the same distribution as ${\mathcal Z}$. Note that $Q$ is not necessarily infinitely ramified, but the main step of this proof is to show there always exists an infinitely ramified law in $\mathbf{R}(P)$.

By the many-to-one formula in Lemma \ref{lem:manytoonePrimal}, for any measurable positive function $f$ and $Q \in \mathbf{R}(P)$, we have
\[
  \int \crochet{\nu, \ex_\theta g} Q(\dd \nu) = \e^{\kappa(\theta)/2} \E(g(\xi_{1/2})).
\]
Let $g$ be any continuous norm-like function with $\E(g(\xi_{1/2}))<\infty$, so that
\begin{equation}\label{e:momQ}
  \sup_{Q \in \mathbf{R}(P)} \int \crochet{\nu, \ex_\theta g} Q(\dd \nu)  <\infty.
\end{equation}
Note that $\mathbf{R}(P)$ is non-empty (as $\mathcal{Z}$ is an infinitely ramified point measure), and closed. Indeed, if $(Q_n)_{n\in\N}$ is a sequence in $\mathbf{R}(P)$
with $\lim_{n\to \infty} Q_n=Q$ in~$\mathbf{P}_\theta$, then from \eqref{e:momQ}, we can apply Lemma \ref{lem:continue++}, hence $Q\in \mathbf{R}(P)$.
Further, using again \eqref{e:momQ}, we see that $\mathbf{R}(P)$ is compact by Lemma~\ref{lem:compact}. 

More generally, writing $Q^{\circledast k}$ for the distribution of the $k$-th generation of a branching random walk with reproduction law $Q$, the same argument shows that if we define for every $n\in\N$
\[
\mathbf{R}_n(P) = \left\{ Q^{\circledast 2^n}: Q\in \mathbf{P}_\theta \text{ and } Q^{\circledast 2^{n+1}} = P \right\},
\]
then 
\[\mathbf{R}(P)=\mathbf{R}_0(P)\supseteq \mathbf{R}_1(P) \supseteq \ldots \supseteq\mathbf{R}_n(P) \supseteq \ldots
\]
form a nested sequence of non-empty compact sets in $\mathbf{P}_{\theta}$. By Cantor's intersection theorem, their intersection is not empty.

This proves that $\mathbf{R}(P)$ always contains the law of some infinitely ramified point measure, say $P^{\circledast 1/2}$, and by iteration,
we construct for every $n\in\N$ the law of an infinitely ramified point measure $P^{\circledast 2^{-n}}$ such that
$$P^{\circledast 2^{-n}} \circledast P^{\circledast 2^{-n}}=P^{\circledast 2^{-n+1}}.$$

For every $n\in \N$, we then consider a branching random walk indexed by $2^{-n}\Z_+$ with reproduction law $P^{\circledast 2^{-n}}$, say $(Z^{(n)}_t: t\in 2^{-n}\Z_+)$. 
The restriction of $Z^{(n+1)}$ to $2^{-n}\Z_+$ has the same law as $Z^{(n)}$, and we can thus construct by Kolmogorov's extension theorem a process $(Z_t: t \in D)$ 
indexed by dyadic rational times, such that for every $n\in\N$, the restriction of $Z$ to $2^{-n}\Z_+$ has the same law as $Z^{(n)}$. That is, $Z$ is a nested branching random walk, 
and by construction, $Z_1$ has the same law as ${\mathcal Z}$. 
\end{proof}

\begin{remark}
With a similar reasoning, one can also prove for instance that an integer-valued random variable having for all $n$ the same distribution as the $n$-th generation of some Galton-Watson branching process can be viewed as the value at time $1$ of a continuous-time branching process. 
\end{remark}

We now conclude this section with the analogue of a well-known result in the theory of infinitely divisible laws: if $(S^{(n)}: n \in \N)$ is a sequence of random walks such that $S^{(n)}_n$ converges, say in distribution, then the limit is infinitely divisible. This is also the case for branching random walks, when we assume convergence in the $\mathcal{P}_\theta$ topology.
\begin{corollary}
\label{prop:22}
Let $(Z^{(n)}:n\in \N)$ be a sequence of branching random walks such that for some $\theta \geq 0$, the law of $Z^{(n)}_n$ converges in ${\mathbf P}_{\theta}$ toward the law of some random point measure ${\mathcal Z}$. If further
\begin{equation}
 \label{eqn:manytoonecv}
 \lim_{n \to \infty} \E(\crochet{Z^{(n)}_n, \ex_\theta }) = \E(\crochet{{\mathcal Z},\ex_\theta })\in(0,\infty),
\end{equation}
then ${\mathcal Z}$ is an infinitely ramified point measure.
\end{corollary}

\begin{proof} The convergence in law of $Z^{(n)}_n$ toward ${\mathcal Z}$ implies that for every $f\in{\mathcal C}_b$, 
$\lim_{n\to \infty} \crochet{Z^{(n)}_n, \ex_{\theta}f} = \crochet{{\mathcal Z}, \ex_{\theta}f}$ in distribution, and by standard arguments of uniform integrability, \eqref{eqn:manytoonecv} then ensures that
$$\lim_{n\to \infty} \E(\crochet{Z^{(n)}_n, \ex_{\theta}f} )= \E(\crochet{{\mathcal Z}, \ex_{\theta}f}). $$

We set $A =\E(\crochet{{\mathcal Z},\ex_\theta }) > 0$ and introduce a random variable $\xi$ such that 
\[
 A \E(f(\xi)) = \E(\crochet{\mathcal{Z},\ex_\theta f})\qquad \text{for all }f \in \mathcal{C}_b.
\]
On the other hand, using the many-to-one formula, we get that for every $n \in \N$, there exists $\kappa_n(\theta) \in \R$ and a random walk $S^{(n)}$ such that
\[
 \E\left( \crochet{Z_n,\ex_\theta f} \right) = \e^{n\kappa_n(\theta)} \E(f(S^{(n)}_n)).
\]
Hence we have $\lim_{n \to \infty} n \kappa_n(\theta) = \ln A$ and $\lim_{n \to \infty} S^{(n)}_n = \xi$ in law. In particular, it follows that $\xi$ is infinitely divisible, hence there exists a L\'evy process $(\xi_t : t \geq 0)$ such that $\xi_1 = \xi$ in law.

This also yields that for all $k \in \N$, 
\begin{equation}
 \label{eqn:mtocv}
 \lim_{n \to \infty} n \kappa_{nk}(\theta) = \frac{\ln A}{k} \quad \text{and} \quad \lim_{n \to \infty} S^{(kn)}_n = \xi_{1/k} \text{ in law.}
\end{equation}
In particular, by Prohorov's theorem, the sequence $(S^{(kn)}_n : n \in \N)$ is tight, and there exists a positive continuous norm-like function $g$ such that
\[
 \sup_{n \in \N} \E\left( \crochet{Z^{(nk)}_n , \ex_\theta g} \right) \leq \sup_{n \in \N} \e^{n \kappa_{nk}(\theta)} \times \sup_{n \in \N} \E(g(S^{(nk)}_n)) < \infty.
\]
Therefore $(Z^{(nk)}_n : n \in \N)$ is tight in ${\mathbf P}_{\theta}$, thanks to Lemma \ref{lem:compact}, and we can extract a subsequence that converges in law, say toward $\tilde{Z}^k$. Moreover, by Lemma~\ref{lem:continue++}, writing $P^{k}_n$ for the law of $Z^{(nk)}_n$, $\tilde{P}^k$ for the law of $\tilde{Z}^k$ and $P$ for the law of $\mathcal{Z}$, we get that
\[
 P = \lim_{n \to \infty} (P^k_n)^{\circledast k} = (\tilde{P}^k)^{\circledast k}.
\]
This result being true for all $k \in \N$, we conclude that $\mathcal{Z}$ is infinitely ramified.
\end{proof}

\section{C\`adl\`ag extension of nested branching random walks and the strong branching property}
\label{sec:infinitelyDivisibleFirstProperties}

Throughout this section, $Z=(Z_t: t\in D)$ denotes a nested branching random walk with $\mathcal{Z}\, \egaldistr \,Z_1$.  In particular, its discrete time skeletons are branching random walks; recall also that we assume \eqref{eqn:integrabilityInfinitelyRamified}.  As in the previous section, we denote by $\xi$ a L\'evy process with characteristic exponent $\Psi$ and reformulate Lemma \ref{lem:manytoonePrimal} as follows.
\begin{lemma} [Many-to-one formula]
\label{lem:manytooneZ} For all $t \in D$ and $f: \R\to \R_+$ measurable, we have
\[
 \E\left( \crochet{Z_t,f} \right) = \E\left( \e^{-\theta \xi_t + t \kappa(\theta)} f(\xi_t) \right).
\]
\end{lemma}

We now prove that the nested branching random walk $Z$ possesses a c\`adl\`ag extension. Recall that $(\calF_t)_{t\in D}$ denotes its canonical filtration, and introduce its right-continuous enlargement
\[
\calF^+_t\coloneqq \bigcap_{s\in D, s>t}\calF_s\,,\qquad t\in\R_+.
\]

\begin{proposition}
\label{prop:cadlag}
Almost surely, there exists a unique extension of $(Z_t: t\in D)$ to a c\`adl\`ag process $(Z_t: t\in \R_+)$ with values in ${\mathcal P}_{\theta}$, and which is further adapted to the filtration $(\calF^+_t)_{t\geq 0}$.

The many-to-one formula of Lemma \ref{lem:manytooneZ} holds more generally for $t\in\R_+$. 
\end{proposition}

\begin{proof} 
Recall from L\'evy's theorem that a sequence $(m_n: n\in\N)$ of finite measures converges weakly if and only if the sequence of Fourier transforms $\hat m_n: {r} \mapsto \crochet{m_n, \ex_{i{r}}}$ converges pointwise to some continuous function $\hat m$. Then $\hat m$ is the Fourier transform of a finite measure $m$ and $\lim_{n\to \infty} m_n=m$ weakly. This shows that a sequence $(\mu_n: n\in\N)$ in ${\mathcal P}_{\theta}$ possesses a limit in 
${\mathcal P}_{\theta}$ if and only if the sequence 
$(\crochet{\mu_n,\ex_{\theta+i{r}}}: n\in \N)$ converges for every ${r}\in \R$ and the limit is a continuous function of ${r}$.

Recall also that for every $t\in D$, we have
$$\E\left(\crochet{Z_{t},\ex_{\theta+i {r}}}\right)= \exp(t\kappa(\theta+i {r})).$$
Using the branching property (B), we see that the process 
\[
M_t({r}) = \e^{-t\kappa(\theta + i{r})}\crochet{Z_t,\ex_{\theta+i{r}}}, \qquad t\in D
\]
is a martingale in dyadic rational times, that is
\[
 \E\left( M_{t}({r}) \middle| \calF_s \right) = M_s({r}) \text{ a.s. for every }  s \leq t, \ s,t\in D.
\]
Further, for each fixed $t\in D$, $M_t(\cdot)$ is the Fourier transform of a random finite measure on $\R$ and is thus continuous a.s. 

Fix $K > 0$ arbitrary, and write $M_t^K$ for the restriction to ${r}\in[-K,K]$ of the function ${r} \mapsto M_t({r})$. So $(M_t^K: t\in D)$ is a martingale in dyadic rational times, taking values in the separable Banach space $\calC([-K,K],\C)$, endowed with the topology of the uniform convergence, and thus possesses a.s. a unique c\`adl\`ag extension $(M_t^K: t\in \R_+)$, which is then a martingale in real time for the filtration $\mathcal{F}_t^+$. See for instance Theorem 3 in \cite{BrDi}.

Plainly, for $K<K'$, $M_t^K$ coincides with the restriction of $M_t^{K'}$ to $[-K,K]$, thus we can define unequivocally
$M_t({r})=M_t^K({r})$ for an arbitrary $K>|{r}|$. The resulting process $(M_t(\cdot): t\in\R_+)$ has a.s. c\`adl\`ag paths in ${\mathcal C}(\R,\C)$, endowed with the topology of uniform convergence on compact intervals. This establishes our first claim. 

Finally, by the martingale property of $(M_t({r}): t\in \R_+)$, we have that, for every $t\in \R_+$:
$$\E(\crochet{Z_t, \ex_{\theta+i{r}}})=\e^{t\kappa(\theta + i{r})}\E(M_t({r}))=\e^{t(\kappa(\theta)+ \Psi({r}))}.
$$
Hence the many-to-one formula of Lemma \ref{lem:manytooneZ} holds with $f=\ex_{\theta+i{r}}$, for all $t\in\R_+$ and ${r}\in\R$. The proof is completed by Fourier inversion.
\end{proof}

Our next goal is to establish a stronger version of the branching property.

\begin{proposition}
\label{prop:strongbranchingproperty}
Let $T$ be an a.s. finite $(\calF^+_t)$-stopping time. On a suitable enlargement of the underlying probability space, there exists an i.i.d. sequence $(Z^n_t: t\in\R_+)_{n\in\N}$ of copies of $(Z_t: t\in\R_+)$, which is independent of $\calF^+_T$, such that
 almost surely
\[
 \forall t \in \R_+, \ Z_{T + t} = \sum_{n=1}^{\infty} \tau_{x_n} Z^n_t,
\]
with $Z_T=\x=(x_1,x_2, \ldots)$. 
\end{proposition}

\begin{remark}
This result, while significantly stronger than assumption (B), is however not the strongest version of the branching property one can look for. In fact, one could establish a version for ``stopping lines'', in the vein of Chauvin~\cite{Chauvin}. But to state this result, one first needs a precise description of the genealogy and the trajectory of individuals, which we are lacking so far. Nonetheless, this result can be proved for branching L\'evy processes, therefore Theorem \ref{thm:main} implies such stronger version of the branching property holds for~$Z$.
\end{remark}

The rest of this section is devoted to the proof of Proposition \ref{prop:strongbranchingproperty}, which relies on a variation of Feller property that we now state. Let $(Z^n_t: t\in\R_+)_{n\in\N}$ denote a sequence of i.i.d. copies of $(Z_t: t\in\R_+)$. For each point measure $\mu\in {\mathcal P}_{\theta}$ and $t\in\R_+$, we consider the random point measure 
$${Y_t}({\mu})\coloneqq \sum_{n\in\N}\tau_{x_n}Z^n_t,$$
where $\x=(x_n: n\in\N)=\mu$.
One checks immediately that $\E (\crochet{{Y_t}({\mu}),\ex_{\theta}})<\infty$, so ${Y_t}({\mu})\in {\mathcal P}_{\theta}$ a.s., and it follows readily from  Lemma~\ref{lem:continue++} that the dependence in $\mu$ is continuous.

\begin{lemma}\label{LFeller}
With the notation above, for every fixed $t\in\R_+$, the process $(Y_t(\mu): \mu\in {\mathcal P}_{\theta})$ is continuous in probability.
\end{lemma}
\begin{proof} Let $(\mu_n: n\in\N)$ be point measures such that $\lim_{n\to \infty} \mu_n=\mu$ in ${\mathcal P}_{\theta}$, that is $\ex_{\theta} \mu_n\Longrightarrow \ex_{\theta} \mu$ as $n\to \infty$, in the sense of weak convergence of finite measures on $\R$. By Prohorov's theorem, $(\ex_{\theta} \mu_n: n\in\N)$ is tight, and thus there exists a positive continuous norm-like function $g:\R\to (0,\infty)$ such that $\sup_{n\in \N} \crochet {\mu_n, \ex_{\theta}g}<\infty$. This enables us to apply Lemma~\ref{lem:continue++} and our conclusion follows. 
\end{proof}

We are now able to establish Proposition \ref{prop:strongbranchingproperty}.

\begin{proof}[Proof of Proposition \ref{prop:strongbranchingproperty}]
Let $T$ be an a.s. finite $(\calF^+_t)$-stopping time. For every $k \in \N$, we set
\[
 T_k \coloneqq 2^{-k} \ceil{2^k T+1}.
\]
So $T_k\geq T+2^{-k}$, $T_k$ is an $(\calF_t)$-stopping time with values in $2^{-k}\N$, and 
$T_k$ decreases to $T$ as $k\to \infty$. 
Next consider an event $A \in \calF^+_T$, and $f: {\mathcal P}_{\theta}\to \R$ a continuous bounded function. 
By right-continuity (see Proposition \ref{prop:cadlag}), we have that for every fixed $t\in\R_+$
$$\E({\bf 1}_{A}f(Z_{T+t}))= \lim_{k\to \infty}\E({\bf 1}_{A}f(Z_{T_k+t})).$$

Let $(Z^n_t: t\in \R_+)_{n\in\N}$ be a sequence of i.i.d. copies of $(Z_t: t\in \R_+)$, which is further independent of $(Z_s: s\in D)$,
and set
$${Y_t}(Z_{T_k})\coloneqq \sum_{n\in\N}\tau_{x_n}Z^n_t,\qquad t\in\R_+,$$
where $(x_n: n\in\N)$ denotes the ranked sequence of the atoms of $Z_{T_k}$.
One checks readily that $A \in \calF_{T_k}$ for every $k\in\N$. Applying (B)
on the event $\{T_k = n2^{-k}\}$ and summing over $n$, we get, provided that $t$ is a dyadic rational number, that
$$\E({\bf 1}_{A}f(Z_{T_k+t}))= \E({\bf 1}_{A}f({Y_t}(Z_{T_k}))).$$
We then let $k\to \infty$ and use Lemma \ref{LFeller} to conclude that
$$\E({\bf 1}_{A}f(Z_{T+t}))= \E({\bf 1}_{A}f({Y_t}(Z_{T}))).$$

In other words, we have shown that the one-dimensional dyadic rational marginals of the conditional distribution of the process $(Z_{T+t}: t\in \R_+) $ given $\calF^+_T$
are the same as those of $({Y_t}(Z_{T}): t\in\R_+)$ given $Z_{T}$. By induction, it follows that the same holds for the finite-dimensional dyadic rational marginals, and since both processes are c\`adl\`ag a.s., we conclude that the conditional distribution of the process $(Z_{T+t}: t\in \R_+) $ given $\calF^+_T$
coincide with that of $({Y_t}(Z_{T}): t\in\R_+)$ given $Z_{T}$. This in turn entails our claim.
\end{proof}

\section{Process with finite birth intensity}
\label{sec:finiteBirthIntensity}

We say that a nested branching random walk $Z$ has a {\it finite birth intensity} if \eqref{eqn:integrabilityInfinitelyRamified} is fulfilled for $\theta=0$. Observe that in this situation, $\ex_0={\mathbf 1}$, and ${\mathcal P}_0$ is simply the space of finite point measures on $\R$, or, equivalently, the space of finite sequences of atoms in $\R$. 

Throughout this section, $(Z_t: t\in\R_+)$ denotes the c\`adl\`ag extension of a nested branching random walk with finite birth intensity. Our goal is to show that the law of this process can be characterized in terms of some basic parameters, and to describe its genealogy. 

\subsection{The first branching time}
\label{subsec:firstBranchingTime}
 The process of the total mass, 
$$\crochet{Z,{\bf 1}}=(\crochet{Z_t,{\bf 1}}: t\in \R_+),$$
takes finite integer values and has c\`adl\`ag paths a.s. The branching property of $Z$ easily transfers to $\crochet{Z,{\bf 1}}$, in the sense that for every $s,t\in\R_+$,
conditionally on $\crochet{Z_s,{\bf 1}}=k$, 
 $\crochet{Z_{t+s},{\bf 1}}$
is independent of $\calF^+_s$ and has the law of the sum of $k$ i.i.d. copies of $\crochet{Z_t,{\bf 1}}$. Using the terminology of Athreya and Ney \cite{AtN}, $\crochet{Z,{\bf 1}}$ is a one-dimensional continuous time Markov branching process, i.e. a Galton-Watson process in continuous time.

In particular, the first branching time 
$$T_B\coloneqq \inf\{t>0: \crochet{Z_t,{\bf 1}}\neq 1\}$$
has an exponential distribution with finite parameter, say $\beta\in \R_+$. For every $0\leq t < T_B$, $Z_t$ possesses a single atom in $\R$; we denote its location by $\zeta_t$ 
(i.e. $Z_t=\delta_{\zeta_t}$),
and declare that $\zeta_t=-\infty$ for $t\geq T_B$. Because $Z$ has c\`adl\`ag paths in ${\mathcal P}_0$, $\zeta$ has also c\`adl\`ag paths during its lifetime $[0,T_B)$. At the branching time $T_B$, we further record the relative positions of the children with respect to that of their parent as the point measure ${\mathbf {\Delta}}$ defined by
\[ \mathbf{\Delta} \coloneqq \tau_{(-\zeta_{T_B-})} Z_{T_B}=\left(x_n - \zeta_{T_B-} : n \in \N\right),\]
where $\x = Z_{T_B}$. We agree for definiteness that ${\mathbf {\Delta}}=\varnothing$ (the zero point measure) when $\beta=0$, that is when $T_B=\infty$ a.s. Note also that $T_B$ is an $(\calF^+_t)$-stopping time.

\begin{lemma} \label{L:LP} In the notation above, the following holds:
\begin{enumerate} \item The process $(\zeta_t: t\in\R_+)$ is a L\'evy process killed at rate $\beta\geq 0$, which is further independent of ${\mathbf {\Delta}}$. 
\item The mass of ${\mathbf {\Delta}}$, i.e. the number of its atoms in $\R$, 
$$\#{\mathbf {\Delta}}\coloneqq {\rm Card}\{i\geq 1: \Delta_i\in\R\}=\crochet{{\mathbf {\Delta}},{\mathbf 1}}$$ 
fulfills
$$\P( \#{\mathbf {\Delta}}=1)=0 \quad \hbox{ and  }\E( \#{\mathbf {\Delta}})<\infty.$$
\end{enumerate}
\end{lemma}

\begin{proof}
As $(\crochet{Z_t,\mathbf{1}} : t \in \R_+)$ is a Galton-Watson process in continuous time, either $T_B<\infty$ a.s. or $\crochet{Z_t,\mathbf{1}} = 1$ a.s. for all $t > 0$. We observe that Lemma~\ref{L:LP} holds trivially when $T_B=\infty$ a.s. Indeed the branching property (B) then simply translates into independence and stationarity of the increments of the trajectory of the only atom $\zeta$. This proves that $\zeta$ is a Lévy process, and the other property comes from $\Delta = \varnothing$ a.s. Therefore, we assume in the rest of the proof that $T_B<\infty$ a.s.

1. The extended version of the branching property in Proposition \ref{prop:strongbranchingproperty} shows in particular that 
for every $t\in\R_+$, conditionally on $t<T_B$, the translated process $(\tau_{-\zeta_t} Z_{t+s}:s\in\R_+)$ is independent of $\calF^+_t$ and has the same law as $(Z_s: s\in \R_+)$. 
Since on that event, ${\mathbf {\Delta}}$ only depends on the translated process, we deduce that for every bounded measurable functionals $F$ and $G$, there is the identity
$$\E(F(\zeta_s: 0\leq s \leq t)G({\mathbf {\Delta}}), t<T_B)=\E(F(\zeta_s: 0\leq s \leq t), t< T_B)\E(G({\mathbf {\Delta}})).$$
Hence ${\mathbf {\Delta}}$ is independent of $(\zeta_t: 0\leq t < T_B)$.

The same argument also shows that conditionally on $t<T_B$, $\zeta$ is a process with independent and stationary increments on the time-interval $[0,t]$, which is further independent of $T_B-t$, as the latter quantity then only depends on the translated process. This proves our assertion that $\zeta$ is a killed L\'evy process. 

2. 
Indeed, the first assertion is plain from the definition of the first branching time $T_B$ and the right-continuity of $Z$. Moreover $\#{\mathbf {\Delta}}=\crochet{Z_{T_B},{\bf 1}}$ is the number of children produced by an individual in the Galton-Watson branching process $\crochet{Z,\mathbf{1}}$. By \cite[Chapter 3, Theorem 2]{AtN}, as $\E(\crochet{Z_1,{\bf 1}})<\infty$, $\#\mathbf{\Delta}$ has finite expectation as well.
\end{proof}

The distribution of the process $Z$ up to and including its first branching time, $(Z_t: 0\leq t\leq T_B)$, is thus determined by the law of the killed L\'evy process $\zeta$ and that of the independent point measure ${\mathbf {\Delta}}$. 
We denote the latter by $\rho$, i.e.
$$\rho(\dd \x)\coloneqq \P({\mathbf {\Delta}}\in\dd \x)$$
and recall that ${\mathbf {\Delta}}$ is never a Dirac point mass.
On the other hand, $\zeta$ shall be viewed as a L\'evy process $\zeta'=(\zeta'(t): t\in\R_+)$ killed at an independent exponential time $T_B$ with parameter $\beta\geq 0$ (recall that $\beta$ is the branching rate of $\crochet{Z,1}$). 
In turn, the law of $\zeta'$ is classically characterized by 
 a triple $(\sigma^2, a',\nu )$, where $\sigma^2\geq 0$ is the Brownian coefficient, $a'\in\R$ the drift coefficient, and $\nu$ the L\'evy measure. The latter is a measure on $\R$ such that $\nu(\{0\})=0$ and $\int(1\wedge x^2)\nu(\dd x)<\infty$. Let $\Phi: \R\to \C$ denote the characteristic exponent of $\zeta'$, which is given by the L\'evy-Khintchin formula:
 \begin{equation}
\label{LKzeta'} \Phi({r})=-\frac{\sigma^2}{2} {r}^2 + ia'{r} +\int_{\R} \left(\e^{i{r} x}-1-i{r} x{\bf 1}_{|x|<1}\right) \nu(\dd x),
\end{equation}
so that for every $t\in\R_+$,
 $$\E(\exp( i{r} \zeta'(t)))=\exp\left(t\Phi({r})\right).$$

Thanks to the strong branching property stated in Proposition \ref{prop:strongbranchingproperty}, 
the law of the full process $(Z_t: t\in\R_+)$ is characterized by that of its restriction to the random time-interval $[0, T_B]$. Indeed, for any $t >0$, there is almost surely only a finite number of reproduction events occurring before time $t$, so using the strong branching property a finite number of times yields the following description.

During the time interval $[0, T_B)$, the point measure $Z_t$ consists of a single individual which starts from $0$ and moves in $\R$ according to $\zeta'$. At time $T_B$, 
which has an exponential distribution with parameter $\beta$ and is independent of $\zeta'$, 
this individual dies at location $\zeta'(T_B)$ and simultaneously gives birth to children at locations
$\zeta'(T_B)+\Delta(1)$, $\zeta'(T_B)+\Delta(2)$, \ldots, where~${\mathbf {\Delta}}=(\Delta(i): i\in\N)$ is a random finite
point measure. More precisely, we know from Lemma \ref{L:LP} that~${\mathbf {\Delta}}$ is independent of $\zeta'$ and the lifetime $T_B$. 
In turn, conditionally on the birth locations, each child evolves after its birth according to the same L\'evy dynamics, independently of the other children. 
At death, these children produce children of their own around their position just before death, according to independent copies of ${\mathbf {\Delta}}$, and so on and so forth.

We stress that this description involves a richer structure than that contained in the sigma-algebra generated by the sole process~$Z$; namely, it is not always possible to recover the ancestral lineages from the latter. Indeed, think for instance of a birth event such that a child is born at the same location as another individual. Then, in general, one cannot discriminate the trajectory of each of these two individuals observing only the process $Z$. 

Putting things together, the distribution of a nested branching random walk with finite birth intensity $Z$ is determined by 
the parameters $(\Phi, \beta,\rho)$, which are henceforth called the parameters of $Z$.

\subsection{Branching L\'evy processes with finite birth intensity}
\label{subsec:branchingLevyProcessFiniteBirthIntensity}

In this section, we introduce formally branching L\'evy processes with finite birth intensity by 
rephrasing technically the verbal description of the dynamics of a nested branching random walk with finite birth intensity. Actually, this is merely an adaptation of Definition 1 in~\cite{BeCF} in a slightly more general setting.

To start with, let $\Phi$ be the characteristic exponent of a L\'evy process, $\beta \geq 0$ and $\rho$ a probability measure on ${\mathcal P}_0^*$, where
$${\mathcal P}_0^*\coloneqq \{\x\in{\mathcal P}_0: \#\x\neq 1\}$$
denotes the subspace of finite point measures which are not Dirac masses. We further suppose that 
$$\int_{{\mathcal P}_0} \#\x \, \rho(\dd \x) <\infty,$$ and that 
$\rho=\delta_{\varnothing}$ if $\beta=0$. 
We stress that we do not assume {\it a priori} that $(\Phi, \beta, \rho)$ is the triple that characterizes the law of some nested branching random walk with finite birth intensity.

 We then need to label individuals and therefore introduce the set of all finite sequences of integers, a.k.a. the Ulam tree,
\[\U = \bigcup_{n \geq 0} \N^n.\]
In particular, the empty sequence $\emptyset$ represents the ancestor\footnote{Beware that we use a different although seemingly similar notation $\varnothing$ for the zero point measure.}. We shall use some further notation in this setting. If $u \in \N^n$, we write
\begin{itemize}
 \item $|u|=n$ the generation of $u$;
 \item $u=(u(1),\ldots, u(n))$, such that $u(k)$ is the $k$-th term of the sequence $u$;
 \item $u_k = (u(1),\ldots, u(k))$ the ancestor at generation $k\leq n $ of $u$;
 \item for $j \in \N$, $u.j = (u(1),\ldots ,u(n),j)$ the $j$-th child of $u$. 
\end{itemize}

Each individual has a birth-time $b_u$, a death-time $d_u$, and a spatial location $\ell_u(t)\in\R$ for $t\in[b_u, d_u)$ which are random and constructed as follows. 
Let $(T_u)_{u\in \U}$, $(\zeta'_u)_{u\in \U}$
and $({\mathbf {\Delta}}_u)_{u\in \U}$ are three independent processes such that:
\begin{itemize}

\item $(T_u)_{u\in \U}$ is a family of i.i.d. exponential variables with parameter $\beta$,

\item $(\zeta'_u)_{u\in \U}$ is a family of i.i.d. L\'evy processes with characteristic exponent $\Phi$ given by \eqref{LKzeta'}, 

\item $({\mathbf {\Delta}}_u)_{u\in \U}$ is a family of i.i.d. random point measures in ${\mathcal P}_0$, each distributed according to $\rho$.
\end{itemize} 
The variable $T_u$ corresponds to the lifetime of the individual labelled by $u$. The birth-time $b_u$ and death-time $d_u$ of this individual are thus given by 
\begin{equation}\label{eq11} b_u=\sum_{j=0}^{|u|-1}T_{u_j} \quad \hbox{and} \quad d_u=b_u+ T_u.
\end{equation}
In turn, the process $\zeta'_u$ governs the motion of the individual $u$ during its life, and ${\mathbf {\Delta}}_u=(\Delta_u(1), \Delta_u(2), \ldots)$ specifies the ranked sequence of the relative positions of the children of $u$ with respect to the location of the individual $u$ at death. Specifically, we have 
$$\ell_u(t)=\zeta'_u(t-b_u) + \sum_{j=0}^{|u|-1}\left( \zeta'_{u_j}(T_{u_j}) + \Delta_{u_j}(u(j+1))\right) , \quad b_u\leq t < d_u. $$
In particular, for $u=\emptyset$, $b_{\emptyset}=0$ and $(\ell_{\emptyset}(t): 0\leq t < d_{\emptyset})$ has the same law as $(\zeta_t: 0\leq t < T_B)$. 

\begin{definition} \label{D:bLpfbi} 
The point measure valued process 
$${Z}_t\coloneqq 
\left( \sum_{u\in\U} {\mathbf 1}_{b_u\leq t < d_u} \delta_{\ell_u(t)}: t\in\R_+\right)$$
(recall our convention that atoms at $-\infty$ are always discarded),
is called a branching L\'evy process with finite birth intensity and parameters $(\Phi, \beta, \rho)$.
\end{definition}
For instance, a binary branching Brownian motion is a branching L\'evy process with finite birth intensity; its parameters are $\Phi({r})= -\frac{1}{2} \sigma^2 {r}^2$, 
$\beta\in\R_+$ and $\rho$ is the Dirac mass at the point measure $\mu=2\delta_0$ (i.e. at $\x=(0,0,-\infty, \ldots)$).
More generally, recall that branching random walks in continuous time, as they were considered first by Uchiyama \cite{Uch} and then by many authors (see e.g. \cite{Kyp1999} and references therein), can be constructed as follows. We first endow the edges of the Ulam tree $\U$ with lengths, such that the length of the edge between a parent $u$ and its child $u.j$ is given by $T_u$ for all $j\in\N$, 
and then assign to that child a weight $\Delta_u(j)$. We assume that the families $(T_u)_{u\in \U}$ and $({\mathbf {\Delta}}_u)_{u\in \U}$ are random and distributed as before. 
The point process obtained at time $t$ by cutting the tree at height $t$ and summing weights on each branch to the root is then a branching random walk in continuous time.
We now see that branching L\'evy processes with finite birth intensity simply result from the superposition of independent L\'evy motions to a branching random walk in continuous time. 

\begin{remark}
We stress that the structure of a branching L\'evy process with finite birth intensity is richer than that of the sigma-algebra generated by the sole point measure process in Definition \ref{D:bLpfbi}, in the sense that by construction, it is equipped with a genealogical tree.
\end{remark}

The next two statements essentially rephrase the second part of Theorem \ref{thm:main} in the case of finite birth intensity. The first claim of the proposition below should be plain from the discussion at the end of the preceding section. In turn, the second claim should be fairly obvious, even if providing full details of the proof would unavoidably be tedious and therefore is left to scrupulous readers.

\begin{proposition} \label{P:Theta0} \begin{enumerate}
\item The c\`adl\`ag extension $(Z_t: t\in\R_+)$ of a nested branching random walk with finite birth intensity and parameters $(\Phi, \beta,\rho)$
is a branching L\'evy process with finite birth intensity and parameters $(\Phi, \beta,\rho)$, and possibly constructed on some enlarged probability space.

\item Conversely, let $\Phi$ be the characteristic exponent of a L\'evy process, $\beta \geq 0$ and $\rho$ a probability measure on ${\mathcal P}_0^*$ with
$$\int_{{\mathcal P}_0} \#\x \, \rho(\dd \x) <\infty,$$ and such that 
$\rho=\delta_{\varnothing}$ if $\beta=0$. 
The restriction to dyadic rational times of a branching L\'evy process with finite birth intensity and parameters $(\Phi, \beta,\rho)$
is a nested branching random walk with finite birth intensity and parameters $(\Phi, \beta,\rho)$. 
\end{enumerate}

\end{proposition}

We now conclude this section by connecting the parameter $(\Phi,\beta,\rho)$ to the functions $\kappa$ and $\Psi$ of the preceding section. 
\begin{lemma}\label{Lkumul} Let $Z$ be a branching L\'evy process with finite birth intensity and parameters $(\Phi,\beta,\rho)$. 
For every $t\in\R_+$ and ${r} \in\R$, we have
$$\E\left(\crochet{Z_t, \ex_{i{r}}}\right)=\exp\left( t\left( \Phi({r}) + \beta \int_{{\mathcal P}_0}\left(\crochet{\mu, \ex_{i{r}}}-1\right)\rho(\dd \mu)\right)\right).$$
\end{lemma}
\begin{proof} In the case when individuals do not move during their lifetimes, that is for branching random walks in continuous time as considered by Uchiyama, we have $\Phi\equiv 0$ 
and the first formula of the statement is easy (it can be read for instance from \cite{Uch} on page 898). The general case then follows from the fact that $Z$ is simply obtained by superposing L\'evy motions with characteristic exponent $\Phi$ to an independent branching random walk in continuous time; see Lemma 2 in \cite{BeCF} for a closely related argument. 
\end{proof}
In the notation of Section~\ref{sec:infinitelyRamifiedPointMeasure}, we can rephrase Lemma \ref{Lkumul} by identifying the cumulant function as
$$\kappa(i{r})= \Phi({r}) + \beta \int_{{\mathcal P}_0}\left(\crochet{\mu, \ex_{i{r}}}-1\right)\rho(\dd \mu),$$
or, equivalently,
$$\kappa(0)=\beta \int_{{\mathcal P}_0}(\#\x-1)\rho(\dd \x)\ ,\ \Psi({r}) = \Phi({r}) + \beta \int_{{\mathcal P}_0}\crochet{\mu, \ex_{i{r}}-1}\rho(\dd \mu).$$

\section{Branching L\'evy processes}
\label{sec:branchingLevyProcess}

Our aim in this section is to get rid of the assumption of finiteness of the birth intensity, that is to consider the case $\theta> 0$. The idea is similar to that of Section~3 in \cite{BeCF} (which actually bears the same title), so we shall merely provide here the main steps without going too far into technical details.

For this purpose, we first consider the case of finite birth intensity treated in the preceding section and introduce an equivalent parametrization.

\subsection{The L\'evy measure for finite birth intensities}
\label{subsec:finiteLevy}
Throughout this section, we consider a branching L\'evy process with finite birth intensity and parameters $(\Phi, \beta, \rho)$ as in the preceding section, and recall the notation $\x=(x_1, \x_2)$ with $\x_2=(x_{n+1}: n\in\N)$ and $\varnothing=(-\infty, \ldots)$. We then define $$\Lambda(\dd \x)\coloneqq \nu(\dd x_1)\delta_{\varnothing}(\dd \x_2)+\beta \rho(\dd \x)\,,\qquad \x\in{\mathcal P}_0.$$
We shall call $\Lambda$ the {\em L\'evy measure} of $Z$. We further set
$$a\coloneqq a'+\beta \int_{{\mathcal P}_0} x_1{\mathbf 1}_{|x_1|<1}\rho(\dd \x).$$

The next statement entails in particular that the parameters $(\Phi, \beta,\rho)$ can be recovered from $(\sigma^2, a, \Lambda)$. Recall that ${\mathcal P}_0^*$ denotes the subspace of finite point measures which are not Dirac point masses. 
 
\begin{lemma}\label{Lparam} The following assertions hold: 
\begin{enumerate}
 \item The branching rate is given by
 $$\beta=\Lambda({\mathcal P}_0^*).$$
 \item For $\beta\neq 0$, we have
$$\rho=\beta^{-1} {\mathbf 1}_{{\mathcal P}_0^*} \Lambda.$$
 \item The L\'evy measure $\Lambda$ fulfills
$$
\Lambda(\{(0,\varnothing)\})=0\ , \ \int_{{\mathcal P}_0}(1\wedge x_1^2)\Lambda(\dd \x)<\infty\  \text{and}\  \int_{{\mathcal P}_0} |\#\x -1| \Lambda(\dd \x)<\infty. 
$$
 \item The characteristic function $\Phi$ is given for every ${r}\in\R$ by
\[
 \Phi({r}) = - \frac{\sigma^2}{2} {r}^2 + i a' {r} + \int_{\mathcal{P}_0\backslash {\mathcal P}_0^*} \left(\e^{i {r} x_1} - 1 - i {r} x_1 \mathbf{1}_{|x_1|<1}\right) \Lambda(\dd \x).
\]
 \item Finally, the cumulant function $\kappa$ is given for every ${r}\in\R$ by 
 $$\kappa(i{r})= -\frac{\sigma^2}{2} {r}^2 + ia{r} + \int_{{\mathcal P}_0} \left(\sum_{n=1}^{\infty} \e^{i{r} x_n}-1-i{r} x_1{\mathbf 1}_{|x_1|<1}\right) \Lambda(\dd \x).$$
\end{enumerate}
\end{lemma}

\begin{proof} The first two assertions follow immediately from the fact that the restriction of $\Lambda$ to ${\mathcal P}_0^*$ is given by $\beta \rho$. The third one then derives from the fact that $\nu(\{0\})=0$ and second point of Lemma \ref{L:LP}. The fourth comes from the fact that $\nu$ is the projection on the first coordinate of $\Lambda - \beta \rho={\mathbf 1}_{{\mathcal P}_0\backslash {\mathcal P}^*_0}$ and the L\'evy-Khintchin formula \eqref{LKzeta'}. Finally, the fifth is merely a translation of Lemma \ref{Lkumul}. 
\end{proof}

\subsection{Nested sequence of branching L\'evy processes}
\label{subsec:nestedLevy}

We now fix $\theta >0$ and consider a measure $\Lambda$ on ${\mathcal P}_{\theta}$ that fulfills the requirements specified in the Introduction, that are
\begin{equation} \label{eq:condlambda1}
 \Lambda(\{(0,\varnothing)\})=0, \quad \int_{{\mathcal P}_{\theta}}(1\wedge x_1^2)\Lambda(\dd \x)<\infty,
 \end{equation}
\begin{equation} \label{eq:condlambda2}
\text{and} \quad \int_{{\mathcal P}_{\theta}} \left( e^{\theta x_1}\ind{x_1>1}+ \sum_{k=2}^{\infty} \e^{\theta x_k} \right) \Lambda(\dd \x) < \infty.
\end{equation}
Roughly speaking, our goal is to define a branching L\'evy process based on the L\'evy measure $\Lambda$. We shall construct the latter as an increasing limit of a nested sequence of 
branching L\'evy processes with finite birth intensity, whose L\'evy measures are given by a suitable truncation of $\Lambda$.

Specifically, for every integer $n\in\N$ and point measure $\mu\in {\mathcal P}_{\theta}$, we write $ \mu^{(n)}$ for the restriction of $\mu$ to $[-n,\infty)$, so the ranked sequence of the atoms of $\mu^{(n)}$ is $\x^{(n)}=(x^{(n)}_i: i\in\N)$, with
$$ x^{(n)}_i=\left\{ \begin{matrix} x_i& \text{ provided }x_i\geq -n\\
-\infty & \text{ otherwise.}
\end{matrix}
\right.$$
One should view the transformations $\mu\mapsto \mu^{(n)}$ for $n\in\N$ as compatible truncations, 
in particular there is the identity
\begin{equation}\label{eq:proj}
\left(\mu^{(n')}\right)^{(n)}=\mu^{(n)}\qquad \text{ for all }n'\geq n.
\end{equation}

We then denote by $\Lambda^{(n)}$ the measure obtained from the image of $\Lambda$ by the map $\mu\mapsto \mu^{(n)}$ by further removing\footnote{Removing this atom is merely an aesthetic matter: keeping it would simply induce fictive birth events, at which the parent gives birth to a single child, exactly at the same location. This would impact the genealogical tree, but not the point measures, and thus can be ignored.} from the latter the atom at $(0,\varnothing)$ if it exists, so that $ \Lambda^{(n)}(\{(0,\varnothing)\})=0$. We may and will view each $\Lambda^{(n)}$ as a measure on the space of finite point measures ${\mathcal P}_0$, rather than on ${\mathcal P}_{\theta}$.

We next observe that for every $n\in\N$, 
$$
 \int_{{\mathcal P}_0}(1\wedge x_1^2)\Lambda^{(n)}(\dd \x)=  \int_{{\mathcal P}_{\theta}}(1\wedge (x_1^{(n)})^2)\Lambda (\dd \x)\leq  \int_{{\mathcal P}_{\theta}}(1\wedge x_1^2)\Lambda(\dd \x) <\infty\ ,
$$
and 
$$  \int_{{\mathcal P}_0} |\#\x -1| \Lambda^{(n)}(\dd \x)\leq \Lambda^{(n)}(\{\varnothing\})+ \e^{\theta n}\int_{{\mathcal P}_{\theta}}\sum_{k=2}^{\infty} \e^{\theta x_k}\Lambda (\dd \x)
<\infty.$$
Hence, we may view each $\Lambda^{(n)}$ as the L\'evy measure of a branching L\'evy process with finite birth intensity. 

The next result claims that one can construct a {\em nested} sequence of branching L\'evy processes with finite birth intensity and L\'evy measures $\Lambda^{(n)}$. Essentially, this follows from the genealogical construction discussed in Section~\ref{subsec:firstBranchingTime} and the compatibility relation \eqref{eq:proj}. We skip details and refer to Section 3 in \cite{BeCF} (see in particular Lemma 3 there) where a similar construction is performed in a less general setting. 

\begin{lemma} \label{L:nestBLP} Let $\sigma^2\geq 0$, $a\in\R$ and $\Lambda$ a measure on ${\mathcal P}_{\theta}$ that fulfills \eqref{eq:condlambda1} and \eqref{eq:condlambda2}. One can construct a sequence $(Z^{(n)}: n\in\N)$ of branching L\'evy processes with finite birth intensity and characteristics
$(\sigma^2, a,\Lambda^{(n)})$, such that for $n\leq n'$, $Z^{(n)}$ results from $Z^{(n')}$ by killing an individual $u$ whenever it makes a negative jump $<-n$,
i.e. when $\Delta \zeta'_u(t)\coloneqq \zeta'_u(t)-\zeta'_u(t-)<-n$, and also killing 
the children $u.j$ which are born at distance greater than $n$ at the left of their parent $u$, i.e. such that $\Delta_u(j)< -n$. 
\end{lemma}

\begin{remark} The killing operation describes above modifies the labelling of individuals and their trajectories. Typically, a birth event for $Z^{(n')}$ at which all the children but one lie at distance greater than $n$ to the left of the parent, is no longer considered as a birth event for $Z^{(n)}$, but rather as an event when an individual makes a jump (of size $\geq -n$) without generating progeny. Nonetheless, if we keep in mind this relabelling of individuals, the ancestral trajectories for $Z^{(n)}$ of course coincides with that for $Z^{(n')}$, and the genealogical tree of $Z^{(n)}$ simply results from the pruning of the genealogical tree of $Z^{(n')}$. 
\end{remark}

Plainly, for every $t\geq 0$, the sequence of atoms of $Z^{(n)}_t$ is a subsequence of that of $Z^{(n')}_t$, or equivalently, in terms of point measures, $Z^{(n)}_t\leq Z^{(n')}_t$. 
This naturally leads us to:

\begin{definition}\label{D:bLp} In the notation of Lemma \ref{L:nestBLP}, the increasing limit 
 $$Z^{(\infty)}_t \coloneqq \lim_{n\to \infty} Z^{(n)}_t,\qquad t\geq 0,$$
is called a {\em branching L\'evy process} with characteristics $(\sigma^2,a,\Lambda)$.
\end{definition}

We easily check that $Z^{(\infty)}_t$ is a point measure in ${\mathcal P}_{\theta}$, a.s.
Indeed, if we write $\kappa^{(n)}$ for the cumulant function of $Z^{(n)}$, so that for $z\in\C$ with $ 0\leq \Re z\leq \theta$,
 $$\E(\crochet{Z^{(n)}_t, \ex_z})=\exp(t\kappa^{(n)}(z)),$$
then by Lemma \ref{Lparam}.5 and analytic continuation, we have
 $$\kappa^{(n)} (z)= \frac{\sigma^2}{2} z^2 + az + \int_{{\mathcal P}_0} \left(\sum_{i=1}^{\infty} \e^{z x_i}{\mathbf 1}_{x_i\geq -n}-1-z x_1{\mathbf 1}_{|x_1|<1}\right) \Lambda(\dd \x).
 $$
 Thanks to \eqref{eq:condlambda1} and \eqref{eq:condlambda2}, we may also define for all $z\in \C$ with $\Re z = \theta$
\begin{equation}\label{eq:kappa}
\kappa^{(\infty)}(z)\coloneqq \frac{\sigma^2}{2} z^2 + az + \int_{{\mathcal P}_{\theta}} \left(\sum_{i=1}^{\infty} \e^{z x_i}-1-z x_1{\mathbf 1}_{|x_1|<1}\right) \Lambda(\dd \x), 
\end{equation}
 and observe that 
$$
\kappa^{(\infty)}(z) = \lim_{n\to \infty} \kappa^{(n)}(z).
$$ 
In particular, we have
\[
 \lim_{n \to \infty} \E\left( \crochet{Z^{(\infty)}_t-Z^{(n)}_t, \ex_{\theta}} \right) = 0,
\]
and since for any ${r} \in \R$
\[
 \E\left( \left|\crochet{Z^{(\infty)}_t, \ex_{\theta+i{r}}} - \crochet{Z^{(n)}_t, \ex_{\theta+i{r}}}\right| \right) \leq \E\left(\crochet{Z^{(\infty)}_t-Z^{(n)}_t, \ex_{\theta}} \right),
\]
we conclude that $\E(\crochet{Z^{(\infty)}_t, \ex_{\theta+i{r}}}) = \e^{t \kappa^{(\infty)}(\theta+i{r})}$.

One can further show that $(Z^{(\infty)}_t: t\in\R_+)$ possesses a c\`adl\`ag version in ${\mathcal P}_{\theta}$ and satisfies the branching property; see Proposition 2 in \cite{BeCF} and its proof for a closely related argument.
Any branching L\'evy process is thus also a nested branching random walk, in the sense that its restriction to dyadic rational times fulfills \eqref{eqn:integrabilityInfinitelyRamified} and (B); more precisely its cumulant function is given by \eqref{eq:kappa}. This corresponds to the second statement of Theorem \ref{thm:main}.

 Our main task in the rest of this work is thus to establish that conversely, every nested branching random walk can be obtained as the restriction to dyadic rational times of some branching L\'evy process.
 
\section{Genealogical structure of a nested branching random walk}
\label{sec:genealogicalStructure}

The purpose of this section is to complete the proof of Theorem \ref{thm:main}, specifically to show that every nested branching random walk arises as the restriction to dyadic rational times of a branching L\'evy process. More precisely, the main result of the section is the following.
\begin{proposition}
\label{prop:derniere}
The càdlàg extension of a nested branching random walk $(Z_t : t \in D)$ satisfying \eqref{eqn:integrabilityInfinitelyRamified} is a branching Lévy process (possibly constructed on some enlarged probability space).
\end{proposition}
Recall that this was established in Section~\ref{sec:finiteBirthIntensity} in the case of processes with finite birth intensity, and that, by construction, a branching L\'evy process is the increasing limit of a sequence of branching L\'evy processes with finite birth intensity. We have to show that similarly, every nested branching random walk is the increasing limit of a sequence of nested branching random walks with finite birth intensities. This will be achieved by showing first that nested branching random walks can be endowed with a {\em natural} \footnote{Beware however that this is by no mean canonical, in the sense that defining the genealogical structure may require some additional randomness, i.e. one may have to work on an enlarged probability space.} genealogical structure. In turn, the genealogical structure will enable us to kill certain atoms depending on the behavior of their ancestral trajectories, and yields the desired approximation by branching processes with finite birth intensities. 

To start with, let us introduce some definitions in this area. A {\em ranked partition} is a sequence $\Pi=(\Pi(j):j \in \N)$ of pairwise disjoint blocks (i.e. subsets) of $\N$. We do {\em not} request the family of blocks $\{\Pi(j):j \in \N\}$ to be a partition of $\N$, the disjoint union of all the blocks $\bigsqcup \Pi(j)$ may be a strict subset of $\N$. 
Next, given some set of times, say $\T\subseteq \R_+$, a {\em genealogical structure} is a family $(\Pi_{s,t}: s,t\in \T\text{ with }s\leq t)$ of ranked partitions which is consistent, in the sense that for all times $r\leq s\leq t$, one has 
\begin{equation} \label{eq:coag}
\forall j\in\N, \ \Pi_{r,t}(j)=\bigsqcup_{i\in \Pi_{r,s}(j)}\Pi_{s,t}(i);
\end{equation} 
we further request that $\Pi_{t,t}$ is simply the ordered partition into singletons (of~$\N$ or $\{1,\ldots,n\}$).
In words, $\Pi_{r,t}(j)$ should be viewed as the block of $\N$ formed by the descent at time $t$ of the individual $j$ at time $r$, and the consistency requirement \eqref{eq:coag} just stresses the plain fact that the latter must coincide with the descent at time $t$ of all individuals at time $s$ which themselves descend from the individual $j$ at time $r$.

\subsection{Natural genealogy of a branching random walk}
\label{subsec:naturalGenealogy}

The purpose of this section is to recall some basic features about discrete genealogies for branching random walks, which will then be useful to construct a genealogical structure in dyadic rational times for nested branching random walks. The presentation is tailored to fit our purpose. 
In this direction, we first recall a construction of branching random walks using for genealogical tree the Ulam tree $\U$.

Let $\rho$ denote the distribution of a random point measure, so we view $\rho$ as the law of some random non-increasing sequence $\x=(x_i: i\in \N)$ in $[-\infty, \infty)$ with $\lim_{i\to \infty} x_i=-\infty$.
We consider 
$( {\x}(u): u\in \U)$ a family indexed by the Ulam tree of i.i.d. copies of ${\x}$. For every $u\in\U$ and $j\in \N$, assign weight $ x_j(u)$ to the vertex $u.j$ and weight $0$ to the ancestor $\emptyset$. For every $u \in \U$, we then write
\[
 X_u = \sum_{j=1}^{|u|} x_{u(j)}(u_{j-1}),
\]
where we recall that $u(j)$ is the $j$-th letter of the word $u$ and $u_j$ the prefix consisting of the $j$ first letters in $u$. For every generation $n\in\Z_+$, we consider the random point measure,
\[
 Z_n = \sum_{u \in \U : |u|=n} \delta_{X_u}.
\]
Further, for all integers $0\leq k\leq \ell$, we define a ranked partition $ \Pi_{k,\ell}$ such that for every $j\in\N$, the block $ \Pi_{k,\ell}(j)$ is given by the ranks at generation $\ell$ of the atoms which descend from the $j$-th largest atom at the $k$-th generation\footnote{If two atoms are at the same position, we order them using the lexicographic order of their index in the genealogical tree $\U$.}. 

This construction yields a branching random walk $ Z=( Z_n: n\in \Z_+)$ with reproduction law $\rho$ and endowed with a genealogical structure $( \Pi_{k,\ell}: 0\leq k \leq \ell)$ that we call {\em natural}, and any branching random walk $Z$ can be obtained by such a construction. Even though the construction is not canonical, in the sense that the family $( {\x}(u): u\in \U)$ cannot always be recovered from $( Z_n: n\in \Z_+)$ and different natural genealogies may sometimes be defined for the same branching random walk, we stress that if $\Pi$ and $\Pi'$ are two natural genealogies of the same branching random walk $Z$, then the pairs $(Z,\Pi)$ and $(Z, \Pi')$ have the same distribution.

For every generation $0\leq k \leq n$ and every $j\in\N$, we write $z_{j,n}$ for the $j$-th largest atom of $Z_n$ and $z_{j,n}(k)$ for its ancestor at generation $k$, defined by
\[
 z_{j,n}(k) = z_{i,k} \quad \mathrm{if} \ j \in \Pi_{k,n}(i).
\]
In particular $z_{j,n}(n)= z_{j,n}$.

To sum up, with this notation, every atom in the branching random walk is uniquely labelled by a couple $(j,n) \in \N \times \Z_+$. The genealogical structure $\Pi$ is a non-anticipative encoding of the genealogical tree of the branching random walk to this labelling. Using this notation, we next recall the pathwise version of the many-to-one identity; see Theorem~1.1 in Shi \cite{ShiSF}.

\begin{lemma}
\label{lem:pathmany21} Let $(Z_n: n\in\Z_+)$ be any branching random walk that fulfills \eqref{eqn:integrabilityInfinitelyRamified} and is endowed with a natural genealogical structure. 
There exists a random walk $S=(S_n: n\in\Z_+)$ such that for every $k \in \N$ and every measurable non-negative function $f:\R^k\to \R$, we have
$$
 \E\left(\sum_{i\geq 1} f(z_{i,k}(1), \ldots, z_{i,k}(k))\right) = \e^{k\kappa(\theta)}\E\left( \e^{-\theta S_k}f(S_1, \ldots, S_k)\right).
$$
\end{lemma}

\subsection{Nested genealogies for nested branching random walks}
\label{subsec:nestedGenealogies}

In this section, we consider a nested branching random walk (in dyadic rational times) $Z=(Z_t: t\in D)$ and shall construct, up to a possible enlargement of the probability space, a natural genealogical structure for this process. Roughly speaking, discrete time skeletons of a nested branching random walk are branching random walks, and our aim is to show that these discrete time skeletons can be equipped with compatible natural genealogies, in the sense that the genealogical structure of a coarser skeleton results from the restriction of the genealogical structure of a finer skeleton.

The construction of a natural genealogical structure for $Z$ relies on the following easy consequence of the branching property (B) combined with the existence of conditional distributions; the proof is straightforward and therefore omitted. 

\begin{lemma}\label{L:condlaw}
\label{lem:genealogyConstruction} Let $(X,Y)$ be a pair of a random point measures distributed as $(Z_{2t},Z_t)$ for some $t\in D$. Up to enlarging the probability space, one can construct a sequence $Y^1, Y^2, \ldots$ of independent copies of $Y=(y_k: k\in\N)$
such that 
$$X=\sum_{k\geq 1} \tau_{y_k} Y^k.$$

\end{lemma}

We start by considering the branching random walk $Z^0\coloneqq(Z_n: n\in \Z_+)$ and use repeatedly (B) to construct (possibly on some enlarged probability space) a natural genealogical structure $\Pi^0=(\Pi^0_{k,\ell}: 0\leq k \leq \ell)$. 
Next, using Lemma \ref{L:condlaw} with $t=1/2$ enables us to construct similarly (possibly on some further enlargement of the probability space), 
 a natural genealogical structure $\Pi^1=(\Pi^1_{k,\ell}: 0\leq k \leq \ell)$ for the branching random walk $Z^1\coloneqq(Z_{n/2}: n\in \Z_+)$ such that $\Pi^1_{2k, 2\ell}=\Pi^0_{k,\ell}$ for all integers $0\leq k\leq \ell$. By iteration, for each $m\in\N$, we can equip the branching random walk $Z^m\coloneqq(Z_{n.2^{-m}}: n\in \Z_+)$ with a natural genealogical structure $\Pi^m$ such that $\Pi^m_{2k, 2\ell}=\Pi^{m-1}_{k,\ell}$ for all integers $0\leq k \leq \ell$. 
This enables us to define unambiguously a genealogical structure in dyadic rational times $\Pi=(\Pi_{r,s}: 0\leq r \leq s \text{ and }r,s\in D)$ by
 $\Pi_{r,s}=\Pi^n_{k,\ell}$ for any integers $k$ and $n$ with $r=k.2^{-n}$ and $s=\ell.2^{-n}$.

We are now able to establish the following pathwise version of the many-to-one formula for nested branching random walks. For every $j\in\N$ and dyadic rational times $s,t\in D$ with $0\leq s \leq t$, we write
$z_{j,t}$ for the $j$-th largest atom of $Z_t$ and $z_{j,t}(s)=z_{k,s}$ with $k$ being the unique integer such that $j\in
\Pi_{s,t}(k)$. Recall also the notation used in Lemma \ref{lem:manytooneZ}; in particular $\xi$ denotes a L\'evy process with characteristic exponent $\Psi$ defined by \eqref{eq:psi}. Similarly, in this notation, every atom in $Z$ is uniquely labelled by a pair $(j,t) \in \N \times D$, and the genealogical structure $\Pi$ encodes the genealogical tree of this labelling.

\begin{lemma} [Pathwise many-to-one formula]
\label{lem:manytooneTrajectorial} The following assertions hold with probability one for every dyadic rational time $t$:

\begin{enumerate}
\item For every $j\in\N$, the trajectory defined for $s\in[0,t]\cap D$ by $s\mapsto z_{j,t}(s)$, possesses a c\`adl\`ag extension to $s\in[0,t]$.

\item For every non-negative measurable functional on the space of c\`adl\`ag paths on $[0,t]$, we have the pathwise many-to-one formula:
\[
 \E\left( \sum_{j \in \N} f(z_{j,t}(s): 0\leq s \leq t) \right) = \E\left( \e^{-\theta \xi_t + t \kappa(\theta)} f(\xi_s: 0\leq s \leq t) \right).
\]
\end{enumerate}
\end{lemma} 

\begin{proof} Consider first a functional $f$ that only depends on the trajectory evaluated at finitely many dyadic rational instants. Then the many-to-one formula in the statement immediately follows from Lemma \ref{lem:pathmany21}. Then an appeal to the monotone class theorem enables us to extend this to all measurable functionals
$f:\R^{D_t}\to \R_+$, where $D_t=[0,t] \cap D$ and $\R^{D_t}$ is endowed with the sigma-algebra generated by the coordinate maps $\omega \mapsto \omega(s)$ for $\omega \in \R^{D_t}$ and $s\in D_t$. 

Since L\'evy processes have c\`adl\`ag paths a.s., our first claim follows taking for $f$ the indicator function of $\{\omega\in \R^{D_t}: \omega \text{ has no c\`adl\`ag extension}\}$ (that this set is indeed measurable is readily seen by considering the number of up-crossings of a path $\omega\in \R^{D_t}$ from $a$ to $b$ with $a<b$ rational numbers). 
The many-to-one formula can then be extended to non-negative measurable functionals on the space of c\`adl\`ag paths on $[0,t]$ by another application of the monotone class theorem. 
\end{proof}

\begin{remark}
Recall from Proposition \ref{prop:cadlag} that $Z$ admits a c\`adl\`ag extension on $\mathcal{P}_\theta$. The first point of the previous lemma gives a slightly distinct statement: to each atom can be associated a c\`adl\`ag ancestral trajectory.
\end{remark}

More generally, a nested branching random walk started from an arbitrary point measure $\mu$ can also be endowed with a natural genealogical structure. 
The case when $\mu=\delta_x$ is a Dirac point mass is of course trivial, as the process is then simply obtained by translating by $x$ all the atoms of $Z$, which does not affect its genealogical structure. In the general case $\mu=\sum_{i=1}^{\infty} \delta_{x_i}$, one needs to combine the genealogies of independent copies of $Z$ started from a Dirac point mass. We refrain from giving a precise description as ordering children from different ancestors would force us to introduce some cumbersome notation. 

We now conclude this section by mentioning that the simple branching property (B) extends to natural genealogies. Specifically, for every $t\in D$, conditionally on $Z_t=\mu$, the process 
$(Z_{t+s}: s\in D)$ equipped with the shifted genealogical structure $(\Pi_{t+r, t+s}: 0\leq r \leq s)$ is independent of $(Z_s: 0\leq s \leq t)$ and $(\Pi_{r,s}: 0\leq r \leq s \leq t)$
and has the same distribution as the nested branching random walk started from $\mu$ and equipped with a natural genealogical structure. Indeed, the same statement holds in the setting of branching random walks, and this easily yields the version for dyadic rational times. 

\subsection{Censoring nested branching random walks}
\label{subsec:censoringBranchingProcesses}

We are now in good shape to construct censored versions of nested branching random walks, roughly speaking by killing individuals at the first time when their ancestral trajectory has a large negative jump. Specifically, fix some threshold $n>0$ and consider for every $t\in D$ the random point measure
$$Z^{(n)}_t\coloneqq \sum_{k=1}^{\infty} {\mathbf 1}_{\{\Delta z_{j,t}(s)>-n \text{ for all } 0\leq s \leq t\}} \delta_{z_{j,t}}\,,$$
where $\Delta z_{j,t}(s)\coloneqq z_{j,t}(s)-z_{j,t}(s-)$ denotes the possible jump at time $s$ of the ancestral trajectory of the $j$-th atom at time $t$. 

\begin{lemma}\label{L:censorbp} For every $b>0$, the censored process $Z^{(n)}\coloneqq (Z^{(n)}_t: t\in D)$ is a nested branching random walk with finite birth intensity.
\end{lemma}

\begin{proof} The branching property extended to natural genealogies that was discussed at the end of the preceding section readily entails that the branching property (B) of $Z$ is transferred to the censored process $Z^{(n)}$. 

Next, using Lemma \ref{lem:manytooneTrajectorial}, we have
\[
 \E\left( \crochet{Z^{(n)}_t,{\mathbf 1}} \right) = \E\left( \e^{-\theta \xi_t + t \kappa(\theta)} \ind{\Delta\xi_s > -n \text{ for all } 0\leq s \leq t} \right).
\]
Recall that $\xi$ is a L\'evy process, denote its L\'evy measure by $\lambda(\dd x)$. Killing $\xi$ at the first instant when it has a negative jump smaller than $-n$ produces a L\'evy process, say $\xi^{\dagger}$,
with L\'evy measure $\lambda^{\dagger}(\dd x)\coloneqq {\mathbf 1}_{\{x>-n\}}\lambda(\dd x)$ and further killed at an independent exponential time with parameter $\lambda((-\infty,-n))$, say $T^{\dagger}$.
Hence we have 
\[
 \E\left( \crochet{Z^{(n)}_t,{\mathbf 1}} \right) = \exp(t\kappa(\theta))\E\left( \e^{-\theta \xi^{\dagger}_t } , t < T^{\dagger} \right).
\]
 The fact that the L\'evy measure $\lambda^{\dagger}$ is zero on $(-\infty,-n)$ ensures the finiteness of the expectation in the right-hand side; see for instance Theorem 25.3 in Sato \cite{Sato}, proving that the censored process $Z^{(n)}$ has finite birth intensity. 
\end{proof}
We now have all the ingredients needed to prove of Proposition \ref{prop:derniere}.

\subsection{Proof of Proposition \ref{prop:derniere}}
\label{subsec:branchingProcAreBranchingLevy}

Let $Z=(Z_t: t\in D)$ be a nested branching random walk endowed with a natural genealogy
and construct the sequence of censored processes $Z^{(n)}$ as in the preceding section. Recall from Lemma \ref{L:censorbp} that the latter are nested branching random walks with finite birth intensity, and thus their c\`adl\`ag extension to real times are branching L\'evy processes, as it was shown in Section~\ref{sec:finiteBirthIntensity}. 
We write $(\sigma_n^2, a_n, \Lambda_n)$ for their characteristics.

Next observe from the construction of the censored processes, that for every $n\leq n'$, 
$Z^{(n)}$ results from $Z^{(n')}$ by killing the children (and of course deleting also their descent) which are born at distance greater than $n$ at the left of their parent,
which is precisely the transformation appearing in Lemma \ref{L:nestBLP}. This entails that $\sigma^{(n)} = \sigma^{(n')}$, $a_n=a_{n'}$, and that 
$\Lambda_n$ coincides with the image of $\Lambda_{n'}$ by the transformation $\x\mapsto \x^{(n)}$ defined in Section~\ref{subsec:branchingLevyProcessFiniteBirthIntensity}, i.e. which consists in sending atoms in $(-\infty,-n)$
to the cemetery state $-\infty$. This enables us to define unambiguously $\sigma^2\coloneqq \sigma_n^2$, $a\coloneqq a_n$ and a measure $\Lambda$ on ${\mathcal P}$ such that, for every $n\in\N$, the image of $\Lambda$ by the transformation $\x\mapsto \x^{(n)}$ is $\Lambda_n$. 

Plainly, $Z^{(n)}_1\leq Z_1$, and \eqref{eqn:integrabilityInfinitelyRamified} thus ensures that 
$$\sup_{n\geq 0} \ln \E\left(\crochet{Z^{(n)}_1, \ex_{\theta}}\right)= \sup_{n\geq 0} \kappa_n(\theta)<\infty,$$ 
where $\kappa_n$ denotes the cumulant of $Z^{(n)}$. Recall that the latter is given for purely imaginary complex numbers by the formula in Lemma \ref{Lparam}.5, so that by analytic continuation, we have
$$\kappa_n(\theta) = \frac{\sigma^2}{2} \theta^2 + a \theta + \int\left(\sum_{i=1}^{\infty} \e^{\theta x^{(n)}_i}-1+\theta x_1{\mathbf 1}_{|x_1|<1}\right) \Lambda(\dd x).$$
By letting $n\to \infty$, we now see that the measure $\Lambda$ has to fulfill the requirements \eqref{eq:condlambda1} and \eqref{eq:condlambda2}.

From Lemma \ref{L:nestBLP} and Definition \ref{D:bLp}, we know that the increasing limit of the sequence of censored processes
$$Z^{(\infty)}_t \coloneqq \lim_{n\to \infty} Z^{(n)}_t \,, \qquad t\geq 0$$
is a branching L\'evy process with characteristics $(\sigma^2, a, \Lambda)$, and it just remains to identify the latter with $Z$ on dyadic rational times. This follows from the monotone convergence theorem and the pathwise many-to-one formula, as
\begin{align*}
 \E(\crochet{Z^{(\infty)}_t,\ex_\theta}) &= \lim_{n \to \infty} \E(\crochet{Z^{(n)},\ex_\theta})\\
 &= \lim_{n \to \infty} \e^{t \kappa(\theta)}\P(\inf_{0\leq s \leq t} \Delta \xi_s > -n) \\
 &= \E(\crochet{Z_t,\ex_\theta}).
\end{align*}
using that $\max_{0\leq s \leq t}|\Delta \xi_s| <\infty$ a.s. and $\e^{t \kappa(\theta)} = \E(\crochet{Z_t,\ex_\theta})$. Since $\crochet{Z_t,\ex_\theta} \geq \crochet{Z_t^{(\infty)},\ex_\theta}$ a.s. for all $t \geq 0$, we conclude that $Z=Z^{(\infty)}$ a.s.

This completes the proof of Proposition \ref{prop:derniere}. We now observe that the first part of Theorem \ref{thm:main} follows from Propositions \ref{thm:irpmToNbrw} and \ref{prop:derniere}.

\begin{remark}
\label{rem:uniqueness}
Note there exists a unique triplet of characteristics $(\sigma^2,a,\Lambda)$ associated to a nested branching random walk $Z$. Indeed, this is the case for branching L\'evy processes with finite birth intensity, hence $(\sigma^2,a,\Lambda_n)$ is uniquely defined for all $n \in \N$. In particular, it shows that $(\sigma^2,a,\Lambda)$ indeed characterizes the law of the branching Lévy process.
\end{remark}

\subsection{Application to infinitely ramified point measures on a half-line}
\label{subsec:branchingSubordinator}

We say that a random point measure ${\mathcal Z}$ is supported on $\R_-$ if ${\mathcal Z}((0,\infty))=0$ a.s., and shall now conclude this article by characterizing infinitely ramified point measures having that property. Theorem \ref{thm:main} entails that the distribution of an infinitely ramified point measure is determined by a triple $(\sigma^2,a,\Lambda)$ with $\sigma^2\geq 0$, $a\in\R$ and $\Lambda$ a measure on ${\mathcal P}$ satisfying \eqref{eqn:integrabilityPi3}, \eqref{eqn:integrabilityPi}, and \eqref{eqn:integrabilityPi2}, which we call therefore the characteristics of ${\mathcal Z}$. Equivalently, $(\sigma^2,a,\Lambda)$ is the characteristic triple of the branching L\'evy process $(Z_t: t\in \R_+)$ such that $Z_1$ has the same law as ${\mathcal Z}$.

\begin{corollary} An infinitely ramified point measure with characteristic triple $(\sigma^2,a,\Lambda)$ is supported on $\R_-$ if and only if the following three conditions are fulfilled:
\begin{itemize}
 \item $\sigma^2=0$ and $\Lambda\left(\left\{\x\in{\mathcal P}_{\theta}: x_1>0\right\}\right)=0$,
 \item $\int_{{\mathcal P}_{\theta}}(1\wedge |x_1|) \Lambda(\dd \x) <\infty$,
 \item $a+\int_{{\mathcal P}_{\theta}} x_1{\mathbf 1}_{\{x_1>-1\}}\Lambda(\dd \x) \leq 0.$
\end{itemize}
\end{corollary} 
\begin{proof} It follows from the many-to-one formula (Lemma \ref{lem:manytooneZ}) that ${\mathcal Z}$ is supported on $\R_-$ if and only if the L\'evy process $\xi$ with characteristic exponent $\Psi$ verifies $\xi_1\leq 0$ a.s., that is, if and only if $-\xi$ is a subordinator. Recall that $\Psi(r)= \kappa(\theta+ir)-\kappa(\theta)$, so \eqref{eq:kappa} shows that 
$$\Psi(r)= -\frac{\sigma^2}{2} r^2 + i a r + \int_{{\mathcal P}_{\theta}}\left( \sum_{i=1}^{\infty} (\e^{irx_i}-1)\e^{\theta x_i} - irx_1{\mathbf 1}_{|x_1|<1}\right)\Lambda(\dd \x).$$
 This enables us to identify the Gaussian coefficient of $\xi$ as $\sigma^2$, and its L\'evy measure, say $\lambda$, as satisfying for any measurable function $f: \R^*\to \R_+$:
$$\int_{\R^*} f(x) \lambda(\dd x)= \int_{{\mathcal P}_{\theta}} \sum_{i=1}^{\infty} f(x_i) \e^{\theta x_i} \Lambda(\dd \x).$$

The fact that $\lambda$ is zero on $(0,\infty)$ when $-\xi$ is a subordinator is thus equivalent to $\Lambda\left(\left\{\x\in{\mathcal P}: x_1>0\right\}\right)=0$. 
Further, the identity 
$$\int_{(-1,0)}|x| \lambda(\dd x) = \int_{{\mathcal P}_{\theta}}\sum_{i=1}^{\infty} \e^{\theta x_i}|x_i|{\mathbf 1}_{\{-1<x_i<0\}}\Lambda(\dd \x)< \infty,$$
and \eqref{eqn:integrabilityPi2} show that the condition $\int_{(-1,0)}|x| \lambda(\dd x)<\infty$ is then equivalent to $\int_{\mathcal P}(1\wedge |x_1|) \Lambda(\dd \x) <\infty$.
When the preceding two requirements hold, the drift coefficient of $\xi$ is given by 
$a+\int_{\mathcal P} x_1{\mathbf 1}_{\{x_1>-1\}}\Lambda(\dd \x)$, and the statement then follows from the characterization of subordinators in the larger class of L\'evy processes. 
\end{proof}

\paragraph*{Acknowledgements.} We thank Pascal Maillard for fruitful discussions, and two anonymous referees for their  careful reading and constructive comments.

\end{document}